\documentclass[12pt]{article}
\usepackage{ct}

\usepackage{navigator}
\embeddedfile{Computation notebook -- version html}{NotebookSymPy_IndependentSetsInCographs.html}
\embeddedfile{Computation notebook -- version ipynb}{NotebookSymPy_IndependentSetsInCographs.ipynb}
\usepackage{bm}
\usepackage{dsfont}
\usepackage{cancel}

\specs{14}{2 (3)}{2022}{}

\dateline{May 21, 2021}{Sep 26, 2022}{Oct 15, 2022}



\title{Linear-sized independent sets in random cographs and increasing subsequences \\ in separable permutations}

\author[1]{Fr\'ed\'erique Bassino}
\author[2]{Mathilde Bouvel\thanks{Supported by the Swiss National Science Foundation, under grants number 200021\_172536 and PCEFP2\_186872.}}
\author[3]{Michael Drmota}
\author[4]{Valentin F\'eray}
\author[5]{Lucas Gerin}
\author[6]{Micka\"{e}l Maazoun\thanks{Supported by EPSRC Fellowship EP/N004833/1.}}
\author[7]{Adeline Pierrot}

\affil[1]{Universit\'e Sorbonne Paris Nord, LIPN, CNRS UMR 7030, F-93430 Villetaneuse, France 

\email{bassino@lipn.fr}
}

\affil[2]{Institut f\"ur Mathematik, Universit\"at Z\"urich, Winterthurerstr. 190, CH-8057 Z\"urich, Switzerland and  Universit\'e de Lorraine, CNRS, Inria, LORIA, F-54000 Nancy, France 

\email{mathilde.bouvel@loria.fr}
}%

\affil[3]{Institute of Discrete Mathematics and Geometry, TU Wien, Wiedner Hauptstr. 8--10, A-1040 Wien, Austria 

\email{michael.drmota@tuwien.ac.at}
}

\affil[4]{Universit\'e de Lorraine, CNRS, IECL, F-54000 Nancy, France 

\email{valentin.feray@univ-lorraine.fr}
}

\affil[5]{CMAP, \'Ecole polytechnique, CNRS, I.P. Paris, 91128 Palaiseau, France 

\email{gerin@cmap.polytechnique.fr}
}

\affil[6]{Department of Statistics, University of Oxford, 24-29 St Giles', Oxford OX1 3LB, UK 

\email{mickael.maazoun@gmail.com}
}

\affil[7]{Universit\'e Paris-Saclay, CNRS, Laboratoire Interdisciplinaire des Sciences du Num\'erique, 91400, Orsay, France 

\email{adeline.pierrot@lri.fr}
}

\def\eps{\varepsilon}
    
\def\si{\sigma}

\def\R{\mathbb{R}}

\def\Leb{\mathsf{Leb}}

\def\cograph{{\sf Cograph}} 
\def\Cograph{{\sf Cograph}}
\def\One{{1}}
\def\Zero{{0}}
\def\stochleq{\le_d}
\def\ind{\widetilde \alpha}
\def\indd{\widetilde {\bm \alpha_2}}
\def\Ind{\alpha}
\DeclareMathOperator{\LIS}{LIS}
\DeclareMathOperator{\LDS}{\omega}

\DeclareMathOperator{\inv}{inv}

\def\proba{\mathbb{P}}
\def\pr{\mathbb{P}}
\def\ex{\mathbb{E}}

\newcommand{\Sample}{\mathbf{Sample}}
\DeclareMathOperator{\Dec}{Dec}

\DeclareMathOperator{\Dirichlet}{Dirichlet}
\DeclareMathOperator{\Law}{Law}

\newcommand{\Exc}{ \scalebox{1.1}{$\mathfrak{e}$}}
\newcommand{\Inversion}{\widetilde{\mathrm{inv}}}
\newcommand{\inversion}{\inv}

\begin{document}

\maketitle


\begin{abstract}
This paper is interested in independent sets (or equivalently, cliques) in uniform random cographs. 
We also study their permutation analogs, namely, increasing subsequences in uniform random separable permutations.

First, we prove that, with high probability as $n$ gets large, the largest independent set 
in a uniform random cograph with $n$ vertices has size $o(n)$. 
This answers a question of Kang, McDiarmid, Reed and Scott. 
Using the connection between graphs and permutations via inversion graphs,
we also give a similar result for the longest increasing subsequence in separable permutations.
These results are proved using the self-similarity of the Brownian limits 
of random cographs and random separable permutations, 
and actually apply more generally to all families of graphs and permutations
with the same limit. 

Second, and unexpectedly given the above results, 
we show that for $\beta >0$ sufficiently small, the expected number of independent sets of size $\beta n$  
in a uniform random cograph with $n$ vertices 
grows exponentially fast with $n$. 
We also prove a permutation analog of this result.
This time the proofs rely on singularity analysis of the associated bivariate generating functions.
\end{abstract}

\section{Introduction}

This paper contains both results for graph and permutation models (connected through the mapping
associating with a permutation its inversion graph); 
for simplicity we present these results and the related backgrounds separately.

\subsection{Independent sets in random cographs}

Cographs were introduced in the seventies by several authors independently (under various names),
see \emph{e.g.}~\cite{Seinsche}.
They enjoy several equivalent characterizations. Among others, \linebreak cographs are  
\begin{itemize}
 \item the graphs avoiding $P_4$ (the path with four vertices) as an induced subgraph;
 \item the graphs whose modular decomposition does not involve any prime graph; 
 \item the inversion graphs of separable permutations;
 \item the graphs which can be constructed from graphs with one vertex by taking disjoint unions and joins.
\end{itemize}
The latter characterization is the most useful for our purpose, let us introduce the terminology.

All graphs considered in this paper are \emph{simple} (\emph{i.e.} without multiple edges, nor loops) and not directed.
Two labeled graphs $(V,E)$ and $(V',E')$ are isomorphic if there exists a bijection
from $V$ to $V'$ which maps $E$ to $E'$.
Equivalence classes of labeled graphs for the above relation are {\em unlabeled graphs}.
Throughout this paper, the \emph{size} of a graph is its number of vertices, and we denote by $V_G$ the set of vertices of any graph $G$.

Let $G=(V,E)$ and $G'=(V',E')$ be labeled graphs with disjoint vertex sets.
We define their \emph{disjoint union} as the graph $(V \uplus V', E \uplus E')$ (the symbol $\uplus$ denoting as usual the disjoint union of two sets). 
We also define their \emph{join} as the graph
 $(V \uplus V', E \uplus E' \uplus (V \times V'))$: 
 namely, we take copies of $G$ and $G'$, 
 and add all edges from a vertex of $G$ to a vertex of $G'$.
Both definitions readily extend to more than two graphs, adding edges between any two vertices originating from different graphs in the case of the join operation.

\begin{definition}
A \emph{labeled cograph} is a labeled graph that can be generated from single-vertex graphs 
 applying join and disjoint union operations. An \emph{unlabeled cograph} is the underlying unlabeled graph of a labeled cograph.
\end{definition}

Recall that, for a given graph $G$, an \emph{independent set} is a subset of vertices in $G$ no two of which are adjacent, while a \emph{clique} is a subset of vertices in $G$  such that every two vertices are adjacent.

The main motivation for studying independent sets in random cographs
comes from the series of papers \cite{ProbabilisticEH,LinearEH}
 on a probabilistic version of the Erd\H{o}s--Hajnal conjecture.

For a graph $G$, a subset of its vertices is called {\em homogeneous}
if it is either a clique or an independent set.
It is well-known that every graph of size $n$ has a homogeneous set of size at least logarithmic in $n$,
and that this is optimal up to a constant
(much work is devoted to get the precise asymptotics; this is
equivalent to the computation of diagonal Ramsey numbers, see \cite{RamseyLowerBound, SahRamsey} for the better bounds up to date).
The conjecture of Erd\H{o}s and Hajnal states that, assuming that the graphs avoid any given subgraph 
(as an induced subgraph),
homogeneous sets of size polynomial in $n$ necessarily exist.
More precisely, for any $H$, there exists a constant~${\eps=\eps(H)>0}$
such that every $H$-free graph has a homogeneous set of size $n^\eps$. 

Despite much effort, the Erd\H{o}s--Hajnal conjecture is still widely open;
see for example the survey \cite{SurveyErdosHajnal}.
A natural relaxation of the conjecture consists in replacing "every $H$-free graph" in the statement above 
by "almost all $H$-free graphs". This weaker version has been established in \cite{ProbabilisticEH}. 
For a large family of constraints $H$, this result was further improved
by Kang, McDiarmid, Reed and Scott in \cite{LinearEH}:
for those $H$, a uniform random $H$-free graph has with high probability a homogeneous set of {\em size linear in $n$}. 
 When this holds, 
the graph $H$ is said to have the {\em asymptotic linear Erd\H{o}s--Hajnal property}
(see \cite{LinearEH} for a formal definition).
Kang, McDiarmid, Reed and Scott ask whether $H=P_4$ has 
the asymptotic linear Erd\H{o}s--Hajnal property,
{\emph i.e.} whether a uniform random cograph with $n$ vertices
has a homogeneous set of size  linear in $n$ \cite[Section 5]{LinearEH}, 
and this question has remained open until now\footnote{
A sketch of proof that $P_4$ does not have the linear Erd\H{o}s--Hajnal property was given in 2012 in the PhD thesis of Andreas W{\"u}rfl~\cite[Chapter 9]{Wuerfl}, as the result of a joint work in progress with C. Hoppen and M. Noy. 
However, this proof sketch contains several approximations or inaccuracies, which would need to be fixed for their argument to be a complete proof. It is not clear whether such fixes are possible.}.

Our first result answers this question in the negative.
In the following, for a graph $G$, we denote $\Ind(G)$ 
the maximum size of an independent set in $G$,
also called the {\em independence number} of $G$.
\begin{theorem}
  \label{thm:sublinear}
  Let $\bm{G_n}$ be a uniform random cograph 
  (either labeled or unlabeled) of size $n$.  
  The maximum size of an independent set in $\bm{G_n}$ is sublinear in $n$, 
  namely $\frac{\Ind(\bm{G_n})}{n}$ converges to $0$ in probability.
\end{theorem}
For a discussion on the difference between the labeled and unlabeled settings,
we refer to \cref{remark:labeling} below.
We recall the standard notation for comparison of random variables: ${\bm X_n=o_P(\bm Y_n)}$ 
if $\bm X_n/\bm Y_n$ tends to $0$ in probability.
The above theorem says that $\Ind(\bm{G_n})$ is $o_P(n)$. 
By taking complements (see the identity \eqref{eq:OmegaEgalAlpha}) 
it also holds that the size of the largest clique in~$\bm{G_n}$ is $o_P(n)$.
Consequently, the size of the largest homogeneous set is also $o_P(n)$,
answering negatively the question of Kang, McDiarmid, Reed and Scott \cite[Section 5]{LinearEH}:
$P_4$ does not have the asymptotic linear Erd\H{o}s--Hajnal property.


A different approach to study independent sets of size  linear in $n$ in random cographs
is the following. 
Let $X_k(G)$ be the number of independent sets of size $k$ in a graph $G$.
From \cref{thm:sublinear},
if $\bm G_n$ is a uniform random (labeled) cograph, 
 then the random variable 
  ${\bm X_{n,k} \! := \! X_{k}(\bm G_n)}$ 
tends to $0$ in probability if $k \sim \beta n$ for some $\beta > 0$ as $n$ tends to infinity.
We show that nonetheless, its expectation grows exponentially fast
for $\beta$ small enough. In particular, this indicates that \cref{thm:sublinear} cannot be proved by a naive use
of the first moment method.
More precisely, we have the following result.

\begin{theorem}\label{th:AsymptX_nalpha}
For each $n\ge 1$, let $\bm G_n$ be a uniform random labeled cograph of size $n$,
and let $\bm X_{n,k}$ be the number
of independent sets of size $k$ in $\bm G_n$.
Then there exist some computable functions $B_\beta >0$, $C_\beta >0$ ($0 < \beta < 1$) 
with the following property.
For every fixed closed interval $[a,b] \subseteq (0,1)$, we have
\begin{equation}\label{eq:AsymptotiqueEX_n,k}
\mathbb{E}[\bm X_{n,k}] \sim B_{k/n} \, n^{-1/2} (C_{k/n})^n
\end{equation}
uniformly for $an \le k \le bn$.
Furthermore,
\begin{enumerate}
\item When $\beta \to 0$, we have
$
C_\beta=1+\beta|\log(\beta)| +\mathrm{o}(\beta\log(\beta)).
$
\item Consequently, there exists $\beta_0>0$ such that $C_\beta >1$ for every $\beta \in (0,\beta_0)$ ;
numerically, we can estimate \[\beta_0 \approx 0.522677\dots\]
\end{enumerate}
\end{theorem}
We have found no explicit formula for the growth constant $C_\beta$ but 
$C_\beta$ can be computed numerically with arbitrary precision: see \cref{eq:CBeta}.
To get an idea of how fast $\bm X_{n,k}$ can grow if $k\sim \beta n$,
we mention that the function $\beta \mapsto C_\beta$
seems to have a unique maximum on $(0,1)$ (see a plot in \cref{fig:C_beta}); 
denoting $\beta^\star$ its location, we have the following numerical estimates:
\[\beta^\star \approx 0.229285\dots; \quad
C_{\beta^\star} \approx 1.366306\dots\]

As additional motivation for \cref{th:AsymptX_nalpha}, let us mention the work of Drmota, Ramos, Requilé and Rué
\cite{DrmotaEtAl_IndependentSetsSubcritical}:
they prove (among other things; see their Corollary 2) 
the exponential growth of the expected number of maximal independent sets
in some subcritical graph classes such as trees, cacti, series-parallel graphs, \dots
(here ``maximal independent sets'' refers to independent sets
that are maximal for inclusion among all independent sets;
such sets are not necessarily of maximum size among all independent sets).
It could be interesting to adapt our 
arguments to consider maximal independent sets in cographs
instead of independent sets of fixed size.


\begin{remark}
\label{remark:labeling}
There are two different ways to pick a uniform random cograph with $n$ vertices: 
taking it uniformly at random among labeled or among unlabeled cographs.
Even if the sizes of independent sets are independent from the labelings,
this gives two different probability distributions, since some unlabeled cographs
have more symmetries and hence fewer distinct labelings than others.

The reader may have noticed that in \cref{thm:sublinear}, we consider either labeled or unlabeled 
uniform random cographs, while \cref{th:AsymptX_nalpha} only considers the labeled setting.
The reason of this choice is given in \cref{ssec:intro_proofs} when discussing proof methods.
\end{remark}

\begin{remark}
A natural question is to determine the order of magnitude of $\Ind(\bm{G}_n)$.
A basic lower bound of order $\sqrt{n}$ is derived as follows.
Since cographs are perfect graphs, for any cograph $G$, we have (see \cite[Theorem 1.4]{SurveyErdosHajnal})
\begin{equation}\label{eq:InequalityIndClique}
\max(\Ind(G),\LDS(G)) \geq \sqrt{n}
\end{equation}
where $\LDS(G)$ is the size of the largest \emph{clique} of $G$.
By symmetry we have that $\Ind(\bm G_n)\stackrel{\text{(d)}}{=}\LDS(\bm G_n)$ 
if ${\bm G_n}$ is a uniform (labeled or unlabeled) cograph. Hence:
\begin{align*}
1&=\mathbb{P}\big(\max(\Ind(\bm G_n),\LDS(\bm G_n)) \geq \sqrt{n}\,\big)\\
&\leq \mathbb{P}\big(\Ind(\bm G_n) \geq \sqrt{n}\,\big)+\mathbb{P}\big(\LDS(\bm G_n) \geq \sqrt{n}\,\big)\\
&=2\mathbb{P}\big(\Ind(\bm G_n) \geq \sqrt{n}\,\big),
\end{align*}
which means that $\Ind(\bm G_n)$ is not $o_P(\sqrt{n})$.

We have not been able to improve this bound, but we believe it to be far from
optimal.
In fact, (limited) numerical simulations, as well as the material in~\cite[Chapter~9]{Wuerfl}
make us believe that $\Ind(\bm G_n)$ is of order $n/\log(n)$.
\end{remark}

\subsection{Increasing subsequences in random separable permutations}
\label{ssec:intro_separables}

The asymptotic behavior of the length of the longest increasing subsequence $\LIS(\bm{s_n})$ in a uniform random permutation $\bm{s_n}$ of size $n$ is an old and famous problem that led to surprising and deep connections with various areas of pure mathematics (representation theory, combinatorics, linear algebra and operator theory, random matrices,\dots).
In particular, it is well-known that~$\LIS(\bm{s_n})$ is typically close to $2\sqrt{n}$
and has Tracy--Widom fluctuations of order $n^{1/6}$.
We refer to \cite{Romik} for a nice and modern introduction to this topic.

Longest increasing subsequences in random permutations in permutation classes are a much newer topic: see \cite{mansour2020LIS_Classes} and references therein.
The methods of the present paper allow the proof of the sublinear behavior of the length of the longest increasing subsequence 
in a uniform random separable permutation. Let us introduce terminology.

Given a permutation $\sigma$ of size $n$ (\emph{i.e.} a sequence $\sigma(1) \dots \sigma(n)$ containing exactly once each integer from $1$ to $n$), 
and given a subset $I = \{i_1 < \dots < i_k\}$ of $\{1, \dots, n\}$, 
the \emph{pattern} of $\sigma$ induced by $I$ is the permutation $\pi$ of size $k$ such that $\pi(\ell) < \pi(m)$ if and only if $\sigma(i_\ell) < \sigma(i_m)$. 
The study of patterns in permutations is an active research topic, 
particularly in enumerative combinatorics, 
see \emph{e.g.} \cite{PermClasses,Kitaev} and references therein. 
The relation ``is a pattern of'' is a partial order on the set of all permutations (of all finite sizes), 
and \emph{permutation classes} are downsets for this order. 
Equivalently, permutation classes can be defined as sets of permutations characterized by the avoidance of a (finite or infinite) set of patterns.

\begin{definition}
A \emph{separable permutation} is a permutation which avoids the patterns $2413$ \linebreak and~$3142$.  
\end{definition}
Separable permutations enjoy many other characterizations, including the following (the related terminology is defined 
later in this paper if needed, or \emph{e.g.} in~\cite{PermClasses}):
\begin{itemize}
 \item they are the permutations whose inversion graph is a cograph; 
 \item they can be obtained from permutations of size $1$ by performing direct sums and skew sums; 
 \item no simple permutation appears in their substitution decomposition. 
\end{itemize}
The class of separable permutations is natural, well-studied, and displays many nice properties; 
we refer the reader to~\cite[end of Section 1.1]{Nous1} for a presentation of these properties and a review of literature.
We shall also review some of them in \cref{Sec:ConstantsSeparablePermutations}. 

We can now state our analog of \cref{thm:sublinear} for separable permutations. 

\begin{theorem}
  \label{thm:sublinear_SeparablePermutations}
  For each $n \ge 1$, 
  let $\bm{\sigma_n}$ be a uniform random separable permutation of size $n$.
  Then, the maximal length of an increasing subsequence in $\bm \si_n$ is sublinear in $n$, namely 
  $\frac{\LIS(\bm{\sigma_n})}{n}$ converges to $0$ in probability.
\end{theorem}

Two remarks about this statement.
First, the above sublinearity result does not only apply to separable permutations, 
but also to any permutation class having a Brownian separable permuton as {\em permuton} limit 
-- see \cref{ssec:intro_proofs}.
Second, as for cographs, we unfortunately did not find a better lower bound for $\LIS(\bm{\sigma_n})$ than the trivial
$\sqrt{n}$ one. The same argument as above applies, where \eqref{eq:InequalityIndClique} is replaced by Erd\H{o}s--Szekeres's Lemma (see \emph{e.g.} \cite[Th.1.2]{Romik}).

We make a further remark about the relation between \cref{thm:sublinear,thm:sublinear_SeparablePermutations}. 
Recall that for any permutation $\sigma$ of size $n$, its \emph{inversion graph} (denoted $\inversion(\sigma)$) is 
the unlabeled version of the graph with vertex set $\{1, \dots, n\}$ where there is an edge between $i$ and $j$ if and only if $i$ and $j$ form an inversion in $\sigma$, that is $(i-j)(\sigma(i)-\sigma(j))<0$.
Clearly, through this correspondence, an increasing sequence in $\sigma$ is mapped to an independent set in $\inversion(\sigma)$. 
Nevertheless, \cref{thm:sublinear_SeparablePermutations} is not simply the translation of \cref{thm:sublinear}
from the graph setting to the permutation setting. 
Indeed, since the inversion graph correspondence is not one-to-one,
for $\bm{\sigma_n}$ a uniform random separable permutation, 
$\inversion({\bm \sigma_n})$ is not a uniform random unlabeled cograph. 
(We further note that defining $\inversion(\sigma)$ as a labeled cograph in the obvious manner, 
$\inversion({\bm \sigma_n})$ would also not be a uniform random labeled cograph.) 


We also establish a counterpart of \cref{th:AsymptX_nalpha} for increasing subsequences in separable permutations. 
\begin{theorem}
\label{thm:expectation_permutations}
    For each $n \ge 1$, 
    let $\bm{\sigma_n}$ be a uniform random separable permutation of size $n$,
    and let $\bm Z_{n,k}$ be the number 
    of increasing subsequences of length $k$ in $\bm \sigma_n$.
    Then there exist some computable functions $D_\beta >0$, $E_\beta >0$ ($0< \beta < 1$)
    with the following property.
For every fixed closed interval $[a,b] \subseteq (0,1)$, we have
\begin{equation}\label{eq:AsymptotiqueEZ_n,k}
\mathbb{E}[\bm Z_{n,k}] \sim D_{k/n} \, n^{-1/2} (E_{k/n})^n
\end{equation}
uniformly for $an \le k \le bn$.
Furthermore,
\begin{enumerate}
\item When $\beta \to 0$, we have
$
E_\beta=1+\beta|\log(\beta)| +\mathrm{o}(\beta\log(\beta)).
$
\item Consequently, there exists $\beta_1>0$ such that $E_\beta >1$ for every $\beta \in (0,\beta_1)$ ;
numerically, we can estimate \[\beta_1 \approx 0.5827\dots\]
\end{enumerate}
\end{theorem}
We observe the same qualitative behavior than for \eqref{eq:AsymptotiqueEX_n,k}: 
$E_\beta$ seems to have a unique maximum (numerically estimated at $\beta\approx 0.2503\dots$).
Moreover, it seems from numerical computations that $E_\beta > C_\beta $ for every $\beta \in(0,1)$ (see \cref{fig:C_beta}). 

\begin{figure}
\begin{center}
\includegraphics[width=7cm]{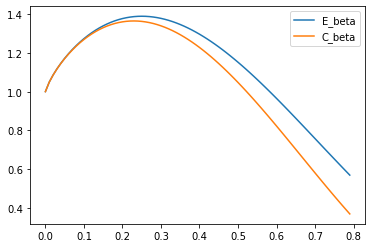}
\caption{Plots of $\beta \mapsto C_\beta$ and $\beta\mapsto E_\beta$.}
\label{fig:C_beta}
\end{center}
\end{figure}

\subsection{Proof methods and universality}
\label{ssec:intro_proofs}
Our sublinearity results are based on limit theorems
for uniform random cographs and uniform random separable permutations.
We first discuss the graph setting.

It is proved in \cite{CographonBrownien,Benedikt}
that a uniform random (labeled or unlabeled) cograph of size $n$
converges in the sense of graphons to a limit $\bm W^{1/2}$,
called the {\em Brownian cographon}
of parameter $1/2$ (see also the independent work of Stufler \cite{Benedikt}).
We refer to \cref{ssec:BrownianCographon} for details.
Moreover, 
the notion of independence number of a graph has been extended to graphons
by Hladk\`y and Rocha \cite{hladky2017independent},
who proved a semicontinuity property for it (see \cref{ssec:IndependenceNbGraphon}).
Combining these two elements, \cref{thm:sublinear} will follow from the fact that
the independence number $\ind(\bm W^{1/2})$ of the  Brownian cographon  
is $0$ a.s (see \cref{ssec:completing_proof}).
To prove the latter, we use the explicit construction of the Brownian cographon
from a Brownian excursion and some self-similarity property of the Brownian excursion
(namely Aldous' decomposition of a Brownian excursion with two independent points
into three independent Brownian excursions of random sizes \cite{Aldous94});
we deduce from that an {\it inequation} in distribution for $\ind(\bm W^{1/2})$ 
(\cref{ssec:autosimilarite}; the use of {\it inequation} instead of {\it inequality}
is justified there)
and we conclude by a fixed point argument (\cref{ssec:FixedPoint}).

An interesting aspect of the proof sketched above is that it relies solely on the fact that
uniform random cographs tend to the Brownian cographon;
moreover the value $p=1/2$ of the parameter of the limit is irrelevant in the proof.
Convergence to the Brownian cographon was proved in \cite{CographonBrownien,Benedikt} both in the labeled
and unlabeled settings, so that \cref{thm:sublinear} is proved simultaneously in both
settings.
In fact, \cref{thm:sublinear} is proved as a special case of the following theorem.
\begin{theorem}
\label{thm:graphes_universel}
Let $\bm G_n$ be a sequence of random graphs tending
to the Brownian cographon $\bm W^p$ for $p \in [0,1)$.
Then the maximum size of an independent set in $\bm G_n$ is sublinear in $n$,
 namely $\frac{\Ind(\bm G_n)}n$ converges to $0$ in probability.
\end{theorem}
By analogy with the realm of permutations (see below), we expect that uniform random
graphs in families of graphs well-behaved for the modular decomposition (\emph{e.g.}, a graph class whose modular decomposition trees are all those obtained from a finite set of prime graphs)
tend to $\bm W^p$, and hence have a sublinear independence number.


Let us now discuss \cref{thm:sublinear_SeparablePermutations},
{\emph i.e.} the sublinearity of the length of the longest increasing subsequence in
a random separable permutation $\bm \si_n$.
It is known that $\bm \si_n$
tends in the permuton topology to a limit $\bm\mu^{1/2}$, called
the {\em Brownian separable permuton} of parameter $1/2$,
see \cite{Nous1} for the original reference.


As discussed earlier, an increasing subsequence of a permutation corresponds to an independent set of its inversion graph. 
We remark in \cref{ssec:Permutons} that if a sequence of permutations converges in distribution to the Brownian separable permuton of parameter $p$, 
then the corresponding inversion graphs converge in distribution to the Brownian cographon $\bm W^p$.

Hence \cref{thm:graphes_universel} implies the following general result, of which \cref{thm:sublinear_SeparablePermutations} is a particular case
(see \cref{ssec:completing_proof} for details).
\begin{theorem}
\label{thm:permutations_universel}
Let $\bm \si_n$ be a sequence of random permutations tending
to the Brownian separable permuton $\bm \mu^p$ for $p \in [0,1)$.
Then the maximal length of an increasing subsequence in $\bm \si_n$ is sublinear in $n$,
 namely $\frac{\LIS(\bm \si_n)}n$ converges to $0$ in probability.
\end{theorem}
We note that the Brownian separable permuton $\bm \mu^p$
has been proved to be a {\em universal limit} for uniform random permutations
in many permutation classes (well-behaved with respect to the substitution decomposition)
\cite{Nous2,Nous3,DecoratedTrees},
so  \cref{thm:permutations_universel} applies to all these classes.


The technique to prove \cref{th:AsymptX_nalpha,thm:expectation_permutations}
is completely different.
Indeed, the expectation of $\bm X_{n,k}$ 
(resp. $\bm Z_{n,k}$) for $k\sim \beta n$ for some $\beta > 0$ is driven by 
a set of cographs (resp. separable permutations) of small probability 
and can therefore not be inferred from their limit in distribution.
In this case, we use the representation of cographs as cotrees, 
and its analogue for separable permutations through substitution decomposition trees.
These tree representations are useful tools in algorithms both for graphs and permutations
(see \emph{e.g.}~\cite{Habib,Bretscher} for graphs and \cite{BBL} for permutations);
in the case of permutations, substitution decomposition trees have also been widely used in recent years
for enumeration problems (see \cite[Section~3.2]{PermClasses} and references therein).
The tree encoding allows us to write a system of equations 
for the bivariate generating function of cographs with a marked independent set
(resp. separable permutations with a marked increasing subsequence).
We then obtain our results through singularity analysis.

Unlike for \cref{thm:sublinear,thm:sublinear_SeparablePermutations},
the results we prove are specific to either labeled cographs
or separable permutations and do not rely on their Brownian limits.
However, our approach should extend to other families of graphs and permutations
well-encoded by their (modular or substitution) decomposition trees,
but we did not pursue this direction.
One such model would be unlabeled cographs:
in this model, the analytic equations involve the so-called P\'olya operators,
making the analysis more technical 
but we do not expect qualitative differences in the result.

\subsection{Organization of the paper}

The proofs of our two sets of results can be read independently.
\begin{itemize}
\item
Section \ref{Sec:GraphonAlpha} provides the necessary background
regarding graphons and permutons.
Then we prove \cref{thm:graphes_universel,thm:permutations_universel} in 
Section \ref{sec:Inequation}.
\item
The proofs of \cref{th:AsymptX_nalpha,thm:expectation_permutations}
are given in \cref{sec:EstimationDrmota} and \cref{Sec:ConstantsSeparablePermutations},
respectively.
\end{itemize}

\section{Preliminaries: graphons, permutons, \\
independence number
and increasing subsequences}\label{Sec:GraphonAlpha}

We first recall some general material from the theory of graphons (\cref{ssec:grapons_basic}).
We present here the strict minimum needed for this paper;
an extensive presentation can be found in \cite[Chapters 7-16]{LovaszBook}.
Then in \cref{ssec:BrownianCographon,ssec:IndependenceNbGraphon,ssec:Permutons} we review recent material from the literature, used for our proof
of \cref{thm:graphes_universel,thm:permutations_universel}:
\begin{itemize}
  \item the convergence of uniform random cographs to the Brownian cographon;
  \item the notion of independence number of graphons;
  \item a connection between graphons and the analogue theory for permutations, that of {\em permutons}.
\end{itemize}
\subsection{Basics on graphons}
\label{ssec:grapons_basic}
A graphon (contraction for {\em graph function}) is a measurable symmetric function from $[0,1]^2$ to~$[0,1]$.
Intuitively, we can think of it as the adjacency matrix of an infinite (weighted) graph with vertex set $[0,1]$.
A finite graph $G$ with vertex set $\{1,\dots,n\}$ can be seen as a graphon $W_G$ as follows:
$W_G(x,y)=1$ if the vertices with labels $\lceil x n \rceil$ and $\lceil y n \rceil$ are connected in $G$
($\lceil z \rceil$ being the nearest integer above $z$, with the unusual convention $\lceil 0 \rceil=1$) and $W_G(x,y)=0$ otherwise.


{\em Sampling.}
Let $W$ be a graphon and $k$ a positive extended integer ({\emph i.e.} $k \in \mathbb Z_{>0} \cup \{+\infty\}$).
We consider two independent families $(\bm U_i)_{1 \le i\le k}$ and $(\bm X_{i,j})_{1\leq i<j \leq k}$
of i.i.d. uniform random variables in $[0,1]$.
Given this, we define a random graph $\Sample_k(W)$ as follows\footnote{In \cite{LovaszBook}, $\Sample_k(W)$ is denoted $\mathbb G(k,W)$.}:
its vertex set is $[k]:=\{1,\dots,k\}$
and for every $i, j$, vertices $i$ and $j$ are connected iff $\bm X_{i,j} \le W(\bm U_i,\bm U_j)$.
In other words vertices $i$ and $j$ are connected
with probability $W(\bm U_i,\bm U_j)$, independently
of each other conditionally on the sequence $(\bm U_i)_{1 \le i\le k}$.

We note that, for $k' > k$, the restriction $\Sample_{k'}(W)[k]$ of $\Sample_{k'}(W)$ to the vertex~set~$[k]$
has the same distribution as $\Sample_{k}(W)$.
In particular, the random graph \linebreak 
 $\Sample_{\infty}(W)$ induces a realization
of all $\Sample_{k}(W)$ in the same probability space.

{\em Convergence.}
By definition, a sequence of graphons $(W_n)$ converges to a graphon $W$
if, for all $k$, $\Sample_{k}(W_n)$ converges in distribution to $\Sample_{k}(W)$.
It can be shown that this is equivalent to the convergence for the so-called cut-distance; see \cite[Theorem 11.5]{LovaszBook}.
We note that the graphon limit is unique only up to some equivalence relation,
called {\em weak equivalence} \cite[Sections 7.3, 10.7, 13.2]{LovaszBook}. Moreover, the quotient of the set of graphons by the weak equivalence relation, equipped with the cut-distance metric, is a compact metric space, that we shall call from now on \textit{the space of graphons}.
Finally, we say that a sequence of graphs~$(G_n)_{n\geq 1}$ converges to a graphon $W$
if the associated graphons $(W_{G_n})$ converge to $W$ in the space of graphons, and that a sequence of random graphs $(\bm G_n)_{n\geq 1}$ converges in distribution to a random graphon $\bm W$, if $W_{\bm G_n}$ converges to $\bm W$ in distribution, as random elements of the space of graphons.

\subsection{Convergence to the Brownian cographon}
\label{ssec:BrownianCographon}

Let $\Exc:[0,1] \to \R$ denote a Brownian excursion of length one.
We recall that, a.s., $\Exc$ has a countable set of local minima, which are all strict and have distinct values\footnote{That $\Exc$ has a.s.~only strict local minima with distinct values
is folklore -- the interested reader may find a proof in \cite[Appendix A]{Nous1}.
This implies readily that the set of local minima is a.s.~countable.}.
Let us denote $\{\bm b_i(\Exc), i \ge 1\}$ an enumeration of the positions of these local minima.
It is possible to choose this enumeration in such a way that the $\bm b_i$'s and the subsequent functions
defined in this section are measurable;
we refer to \cite[Lemma 2.3]{MickaelConstruction} and \cite[Section 4]{CographonBrownien} for details.

We now choose i.i.d. Bernoulli variables $\bm s_i$ with $\proba(\bm s_i = 0) = p$, independent from the foregoing, and write $\bm S^p=(\bm s_i)_{i \ge 1}$.
We call $(\Exc,\bm S^p)$ a \textit{decorated Brownian excursion}, 
thinking of the variable $\bm s_i$ as a decoration attached to the local minimum $\bm b_i(\Exc)$.

For $x,y\in[0,1]$, we define $\Dec(x,y;\Exc,\bm S^p)$ to be
the decoration of the minimum of $\Exc$ on the interval $[x,y]$ (or $[y,x]$ if $y\le x$; we shall not repeat this precision below).
If this minimum is not unique or attained in $x$ or $y$ and therefore not a local minimum,
$\Dec(x,y;\Exc,\bm S^p)$ is ill-defined and we take the convention $\Dec(x,y;\Exc,\bm S^p)=\Zero$.
Note however that, for uniform random $x$ and $y$, this happens with probability $0$,
so that the object constructed in \cref{def:BrownianCographon} below is independent from this convention.
\begin{definition}\label{def:BrownianCographon}
	The Brownian cographon $\bm W^p$ of parameter $p$ is the random function
    $$
\begin{array}{ r c c c}
\bm W^p: & [0,1]^2 &\to& \{0,1\};\\
 & (x,y) & \mapsto & \Dec(x,y;\Exc,\bm S^p).
\end{array}
$$
\end{definition}

For example, in  Fig.\ref{fig:DecompoAldous} if the decoration
at $b$  is $0$ (resp. $1$), then the graphon is constant equal to $0$ (resp. $1$) on the rectangle $(a,b) \times (b,c)$.

The following was proved independently in \cite{Benedikt} and \cite{CographonBrownien}.
\begin{theorem}\label{th:BrownianCograph}
  Uniform random cographs (either labeled or unlabeled) converge in distribution
  to the Brownian cographon of parameter $1/2$, in the space of graphons.
\end{theorem}

\subsection{Independence number of a graphon and semi-continuity}
\label{ssec:IndependenceNbGraphon}
Let $W$ be a (deterministic) graphon.
Following Hladk\`y and Rocha \cite{hladky2017independent}, we define an independent set $I$ of a graphon $W$
as a measurable subset of $[0,1]$ such that $W(x,y)=0$
for almost every $(x,y)$ in $I \times I$.
The {\em independence number} $\ind(W)$ of a graphon $W$ 
is then 
\begin{equation}
  \ind(W) = \sup_{I \subset [0,1] \atop I\text{ independent set of }W} \Leb(I),
  \label{eq:DefInd}
\end{equation}
where $\Leb(I)$ denotes the Lebegue measure of $I$.
Note that  $\ind(W)$ is attained by some independent set
  $I$ (that is, the supremum in Eq.~\eqref{eq:DefInd} is in fact a maximum):
  this follows from {\it e.g.} \cite[Lemma 2.4]{hladky2019Independent}.

Clearly, for a graph $G$ we have 
\begin{equation}\label{eq:alpha_tilde_vrai_graphe}
\ind(W_G)=\Ind(G)/|V_G|,
\end{equation}
where $\Ind(G)$ is the maximum size of an independent set in $G$.

Of crucial interest for this paper is the lower semi-continuity of the function
$\ind$ on the space of graphons \cite[Corollary 7]{hladky2017independent}.
 Concretely,
this says the following.
\begin{proposition}
  \label{prop:SemiContinuity}
  Suppose that $(W_n)_{n \ge 1}$ is a sequence of graphons that converges to some $W$ in the space of graphons.
  Then
  $\limsup \ind(W_n) \le \ind(W)$.
\end{proposition}

\begin{remark}
In the following, we will consider a random variable
of the kind $\ind(\bm W)$ where $\bm W$ is a random graphon.
For this to make sense, the map $\ind$ should be measurable.
Since it is defined as a supremum over an uncountable set,
its mesurability is not a priori clear. However, it is known that
any semi-continuous function is measurable, so that \cref{prop:SemiContinuity}
implies that $\ind$ is indeed measurable.
We shall not discuss this point further in the paper.
\end{remark}

In the rest of this subsection, we give an alternative definition for $\ind(W)$.
This definition is not needed in the rest of the paper 
(and therefore can be safely skipped);
however, it answers a question raised by Hladk\`y and Rocha \cite[Section 3.2]{hladky2017independent},
who asked for a connection between the statistics $\ind(W)$ 
and subgraph densities (or equivalently, samples) of $W$.

For a graphon $W$, we set
\begin{equation}
\indd(W) = \liminf_{k\to\infty}\ \frac 1k \, \Ind(\Sample_\infty(W)[k]).
\end{equation}
Since $\Sample_\infty(W)$ is a random graph, the right-hand side is a priori a random variable.
We recall that $\Sample_\infty(W)$ is constructed from i.i.d. random variables $\{\bm U_i,\bm X_{i,j},\, 1\leq i<j\}$.
We denote $\mathcal{G}_n$ the $\sigma$-algebra
generated by $\{\bm U_i, \bm X_{i,j},\, n<i<j\}$.
It is a simple exercise to see that $\indd(W)$ 
is measurable with respect to the {\em tail $\sigma$-algebra} $\bigcap_{n \ge 1} \mathcal{G}_n$.
By Kolmogorov's $0-1$ law (easily adapted to our situation with bi-indexed i.i.d. random variables),
$\indd(W)$ is  almost surely equal to a constant.

\begin{lemma}\label{lem:alpha2=alpha}
  For any graphon $W$, we have $\indd(W)=\ind(W)$ almost surely, and the $\liminf$ defining $\indd(W)$ is almost surely an actual limit.
\end{lemma}

\begin{proof}
We first prove $\indd(W) \ge \ind(W)$ almost surely.
Let $I$ be an independent set of $W$.
For any $k\geq 1$, we observe that the set $\bm J_k:=\{j \le k: \bm U_j \in I\}$ 
is a.s. an independent set of $\Sample_\infty(W)[k]$. 

Hence, a.s.
  \[ \frac 1k \Ind(\Sample_\infty(W)[k]) \ge \frac 1k |\bm J_k|.\]
  As $k$ tends to infinity, the law of large numbers asserts that $|\bm J_k|/k$ tends a.s. to $\Leb(I)$.
  Therefore we have a.s. $\indd(W) \ge \Leb(I)$. 
  Since this holds for any independent set $I$ of $W$, we can consider the independant set $I$ that realizes the maximum in Eq.~\eqref{eq:DefInd},
  proving $\indd(W) \ge \ind(W)$ a.s..

  Let us prove the converse inequality.
  It is known that $(\Sample_\infty(W)[k])$ converges a.s.
   to $W$ in the space of graphons
    (e.g. as a consequence of \cite[Lemma 10.16]{LovaszBook}).
  Using \eqref{eq:alpha_tilde_vrai_graphe} and \cref{prop:SemiContinuity} this implies that, a.s.,
  \[ \indd(W) \le \limsup_{k\to\infty} \frac 1k \Ind(\Sample_\infty(W)[k])
  = \limsup_{k\to\infty} \ind(\Sample_\infty(W)[k]) \le \ind(W).\]
  This concludes the proof that 
  almost surely $\indd(W)=\ind(W)$. Moreover in the identity \linebreak $\indd(W) = \liminf_{k\to\infty}\,\frac 1k \, \Ind(\Sample_\infty(W)[k])$ the $\liminf$ is an actual limit. 
  \end{proof}

\subsection{The Brownian separable permuton and its relation to the Brownian cographon}
\label{ssec:Permutons}
The theory of permutons (see \cite{FinitelyForcible,Permutons}) plays the same role for limits of permutations as the theory of graphons does for dense graphs. A permuton is a probability measure on the unit square with uniform marginals, and the space of permutons equipped with the weak convergence of measures is a compact metric space. 
We attach to each permutation $\sigma$ of size $n\geq 1$ the measure $\mu_\sigma$ on the unit square with density $(x,y)\mapsto n \mathds 1_{\sigma(\lceil nx \rceil) = \lceil ny \rceil}$, which is a permuton.
This defines a dense embedding of the set of permutations into the space of permutons.

Recall that $\inversion(\sigma)$ denotes the (unlabeled) inversion graph of a permutation $\sigma$. 

\begin{proposition}
\label{prop:inv}
	Let $p \in [0,1]$ and $(\bm \sigma_n)_n$ be a sequence of random permutations such \linebreak that $\mu_{\bm \sigma_n} \xrightarrow[n\to\infty]{d} \bm \mu^p$, 
	where $\bm \mu^p$ is the Brownian separable permuton of parameter $p$ defined in \cite[Definition 3.5]{Nous3}.
	Let $\bm G_n = \inversion(\bm \sigma_n)$. 
	Then we have the convergence in distribution $W_{\bm G_n} \xrightarrow[n\to\infty]{d} \bm W^p$ in the space of graphons,
	where $\bm W^p$ is the Brownian cographon of parameter~$p$.
\end{proposition}

\begin{remark}
	It was observed in \cite[End of Section 2]{FinitelyForcible} that $\inversion$ possesses an extension which is a continuous map $\Inversion$ from the space of permutons to the space of graphons. 
	The above proposition implies that the image of $\bm \mu^p$ by $\Inversion$ is $\bm W^p$.
\end{remark}

\begin{proof}
	For every $k \geq 1$, denote $\bm b_{k,p}$ a uniform random plane binary tree with $k$ (unlabeled) leaves, whose internal vertices are decorated with independent signs $\{\oplus,\ominus\}$ such that $\proba(\oplus)=p$.
	Before entering the actual proof, we present a useful link between a separable permutation and an unlabeled cograph constructed from  $\bm b_{k,p}$. 
	
	Following \cite[Definition 2.3]{Nous3}, we may associate with $\bm b_{k,p}$ a separable permutation, denoted $\mathrm{perm}(\bm b_{k,p})$.  
	We do not recall this construction here (for details, see the above reference or the beginning of \cref{Sec:ConstantsSeparablePermutations}), 
	but indicate an important property it enjoys: 
	for $1\leq i<j \leq k$, we have $\mathrm{perm}(\bm b_{k,p})(i) > \mathrm{perm}(\bm b_{k,p})(j)$ 
	if and only if the youngest common ancestor of the $i$-th and $j$-th leaves (in the left-to-right order) of $\bm b_{k,p}$ carries a $\ominus$ sign. 
	
	Similarly, we may also associate with $\bm b_{k,p}$ an unlabeled cograph. 
	We first replace $\ominus$ by $\One$ and $\oplus$ by $\Zero$ in all internal nodes
	and then we forget the plane embedding.
	We denote by $\tilde{\bm b}_{k,p}$ the resulting non-plane and unlabeled decorated tree.
	With this tree, we associate an unlabeled cograph $\Cograph(\tilde{\bm b}_{k,p})$ as follows:
	its vertices correspond to the leaves of $\tilde{\bm b}_{k,p}$, and there is an edge between the vertices corresponding to leaves $\ell$ and $\ell'$ 
	if and only if the youngest common ancestor of $\ell$ and $\ell'$ carries the decoration $\One$.
        An alternative recursive presentation of this construction,
	making it clear that the constructed graph is indeed a cograph,
	is given at the beginning of \cref{sec:EstimationDrmota}.
	
	By construction, the equality $\inversion(\mathrm{perm}(\bm b_{k,p})) = \Cograph(\tilde{\bm b}_{k,p})$ 
	of unlabeled graphs holds.
	
	Denote $\bm \sigma_{n,k}$ a uniform random pattern of size $k$ in $\bm \sigma_n$.
	Theorem 3.1 and Definition 3.5 in \cite{Nous3} imply that $\bm \sigma_{n,k}$ converges in distribution to the random separable permutation $\mathrm{perm}(\bm b_{k,p})$. 
	As this is a convergence in distribution in the discrete space consisting of all permutations of size $k$, the map $\inv$ is continuous, and 
	we obtain the following convergence of unlabeled graphs:  
	\begin{equation}
	\label{eq:inv_si_nk}
	\inversion(\bm \sigma_{n,k})\xrightarrow[n\to\infty]{d} \inversion(\mathrm{perm}(\bm b_{k,p})).
	\end{equation}
	
	It is easy to check that the actions of taking patterns (resp. induced subgraphs) 
	and of computing inversion graphs commute. 
	Namely, for a permutation $\sigma$ and a subset $I$ of its indices, 
	the inversion graph of the pattern of $\sigma$ induced by $I$ 
	is the subgraph of the inversion graph of $\sigma$ induced by the vertices corresponding to $I$.
	Therefore, $\inversion(\bm \sigma_{n,k})$ -- which appears on the left-hand-side of \cref{eq:inv_si_nk} -- has the same distribution as 
	the subgraph induced by a uniform random subset of $k$ distinct vertices of $\bm G_n$.
	
	On the right-hand side of \cref{eq:inv_si_nk}, we have already identified $\inversion(\mathrm{perm}(\bm b_{k,p}))$ as $\Cograph(\tilde{\bm b}_{k,p})$.
	We recall that $\tilde{\bm b}_{k,p}$ is the non-plane version
	of a uniform random (unlabeled) plane binary tree with independent decorations on its internal nodes.
	We claim that this has the same distribution as the unlabeled version
	of a uniform random labeled non-plane binary tree,
	with the same rule for  decorations of the internal nodes
	(which we denote $\bm b^{\not P,L}_{k,p}$).
	Admitting this claim for the moment, and comparing with 
	\cite[Proposition 4.3]{CographonBrownien},
	we get that the right-hand side of \cref{eq:inv_si_nk}
	is distributed as $\Sample_k( \bm W^p)$.
	
	With these considerations in hand, we can use \cite[Theorem 3.8]{CographonBrownien}
	(more precisely the implication $(d) \Rightarrow (a)$ and Eq.\,(4) following this theorem) and conclude from  \cref{eq:inv_si_nk}
	that $W_{\bm G_n} \xrightarrow[n\to\infty]{d} \bm W^p$.
	This ends the proof of the proposition, up to the above claim.
	
	It remains to prove that $\tilde{\bm b}_{k,p} \stackrel{d}{=} \bm b^{\not P,L}_{k,p}$, as non-plane unlabeled trees.
	Since the rule for the random decorations are the same on both sides, we disregard decorations, 
	and denote the underlying undecorated random trees $\tilde{\bm b}_{k}$ and $\bm b^{\not P,L}_{k}$ respectively.
	To prove that $\tilde{\bm b}_{k}\stackrel{d}{=}\bm b^{\not P,L}_{k}$, 
	we compare both distributions with that of a uniform labeled plane binary tree with $k$ leaves $\bm b^{P,L}_{k}$.
	Since every non-plane labeled binary tree with $k$ leaves can be embedded in the plane in $2^{k-1}$ ways,
	we have $\bm b^{\not P,L}_{k} \stackrel{d}{=} \bm b^{P,L}_{k}$ 
	as non-plane unlabeled trees
	(there are no symmetry problems, since trees are labeled).
	On the other hand, since every plane unlabeled binary tree with $k$ leaves can be labeled in $k!$ ways,
	we have $\tilde{\bm b}_{k} \stackrel{d}{=} \bm b^{P,L}_{k}$
	as non-plane unlabeled trees
	(again, there are no symmetry problems, since trees are plane).
	We conclude that $\tilde{\bm b}_{k,p} \stackrel{d}{=} \bm b^{\not P,L}_{k,p}$,
	as wanted.
\end{proof}
    
  \section{Proof of the sublinearity results through self-similarity}\label{sec:Inequation}

The main part of the proof of our sublinearity results (\cref{thm:graphes_universel,thm:permutations_universel})
is done in the continuous world,
proving that the independence number $\ind(\bm W^{p})$ of the  Brownian cographon is almost surely equal to $0$.
To this end, we first show 
that the distribution of $\ind(\bm W^{p})$ 
is solution of a specific inequation 
-- this is \cref{prop:InequationSatisfied}. 
Next, we prove that the only solution of this inequation is the Dirac distribution $\delta_0$ -- this is \cref{prop:SolvingTheInequation}. 
All results are gathered in \cref{ssec:completing_proof},
completing the proofs of \cref{thm:graphes_universel,thm:permutations_universel}.
  
    \subsection{An inequation in distribution}
  \label{ssec:autosimilarite}
 We use the standard {\em stochastic domination order}
  between real distributions $\mu$ and $\nu$.
  Namely, we write $\mu \stochleq \nu$
   if $\mu([x,+\infty)) \leq \nu([x,+\infty))$ for every real $x$.
By Strassen's Theorem, this is equivalent to the fact that we can find ${\bm Z_1}$ and ${\bm Z_2}$ 
  defined on the same probability space 
  with distributions $\mu$ and $\nu$ respectively,
  such that ${\bm Z_1} \le {\bm Z_2}$ almost surely. 
  
  Our goal is now to show that the distribution of the random variable $\ind(\bm W^{p})$
  is stochastically dominated by another distribution
   defined using some independent copies of $\ind(\bm W^{p})$
   (we refer to such a relation as an {\em inequation in distribution}\footnote{We
    use the term {\em inequation} and not inequality,
    because the upper bound also involves the distribution of $\ind(\bm W^{p})$.}).
   To this end, 
   we use Aldous' decomposition of a Brownian excursion
   into three independent excursions
   (see \cref{fig:DecompoAldous}).
   This decomposition has an immediate counterpart, where we decompose
   a Brownian cographon into three independent Brownian cographons.
   We then look closely at the behavior of the functional $\ind$
   along this decomposition.
   
   We introduce the notation needed to state this inequation in distribution.
  For a random variable ${\bm Y}$, let us denote by $\Law({\bm Y})$ its distribution.
  Recall that, for positive real numbers $\alpha_1, \dots, \alpha_k$, 
  the Dirichlet distribution $\Dirichlet(\alpha_1,\dots,\alpha_k)$
  is a probability measure on the simplex $\{(x_1,\cdots,x_k): x_1+\cdots+x_k=1, x_i \geq 0 \text{ for all } i\}$:
  by definition it has density proportional to $\prod_{i\le k} x_i^{\alpha_i}$ with respect to the Lebesgue measure.

   Let $\mu$ be a probability distribution on $[0,1]$ and $p$ a parameter in $[0,1]$.
   We define the following random variables:
    \begin{itemize}
      \item $({\bm \Delta_0},{\bm \Delta_1},{\bm \Delta_2})$ is a random vector in $[0,1]^3$
        with distribution $\Dirichlet(1/2,1/2,1/2)$;
      \item $\bm X^\mu_0$, $\bm X^\mu_1$ and $\bm X^\mu_2$ are three independent random variables
             with distribution $\mu$, and independent from $({\bm \Delta_0},{\bm \Delta_1},{\bm \Delta_2})$;
           \item $ \bm B$ is a Bernoulli$(1-p)$ random variable, independent from $({\bm \Delta_0},{\bm \Delta_1},{\bm \Delta_2},\bm X^\mu_0,\bm X^\mu_1,\bm X^\mu_2)$;
      \item finally, we set 
    \begin{align}
      {\bm Y^\mu_0}&= {\bm \Delta_0} \, \bm X^\mu_0 + {\bm \Delta_1} \,  \bm X^\mu_1 + {\bm \Delta_2} \, \bm X^\mu_2 \label{eq:DefY0}\\
      {\bm Y^\mu_1}&= {\bm \Delta_0} \, \bm X^\mu_0 + \max({\bm \Delta_1} \,  \bm X^\mu_1, {\bm \Delta_2} \, \bm X^\mu_2)
      \label{eq:DefY1}\\
      {\bm Y^\mu_{(p)}}&=  \bm B {\bm Y_1^{\mu}} + (1- \bm B) {\bm Y_0^{\mu}} \label{eq:DefY} \\
      \Phi_p(\mu)&= \Law({\bm Y^{\mu}_{(p)}})
      \nonumber
    \end{align}
    \end{itemize}

    Then the inequation we are interested in is 
    \begin{equation}
      \mu \, \stochleq\,  \Phi_p(\mu).
      \label{eq:TheDistributionalInequality}
    \end{equation}

  \begin{proposition}
    \label{prop:InequationSatisfied}
    For $p\in [0,1]$, the distribution $\mu$ of $\ind(\bm W^{p})$ satisfies the inequation \eqref{eq:TheDistributionalInequality}.
  \end{proposition}
  
  	\begin{figure}[t]
  	\begin{center}
  	\includegraphics[width=12cm]{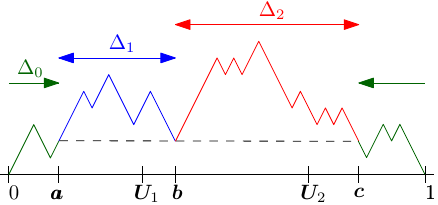}
  	\caption{A stylized version of a Brownian excursion and the corresponding $\bm a,\bm b,\bm c,\bm \Delta_0,\bm \Delta_1,\bm \Delta_2$.}
  	\label{fig:DecompoAldous}
  	\end{center}
  	\end{figure}

    \begin{proof}
	Fix $p\in [0,1]$ and let $(\Exc,\bm S^p)$ be a decorated Brownian excursion.  
	Let $\bm U_1< \bm U_2$ be a reordered pair of independent and uniform random variables on $[0,1]$, chosen independently from $(\Exc,\bm S^p)$. 
	Almost surely, the function $\Exc$ reaches its minimum on $[\bm U_1, \bm U_2]$ exactly once, and at a local minimum.
	Let us denote by $\bm s$ the sign of this local minimum in $\bm S^p$. 
	Let $\bm b$ be the position where this local minimum is reached (see \cref{fig:DecompoAldous}). Let also 
  	\[ \bm a = \max\{t\leq \bm U_1, \Exc(t) = \Exc( \bm b)\},\quad
  	 \bm c = \min\{t\geq \bm U_2, \Exc(t) = \Exc( \bm b)\}.
  	\]
  	Set 
  	\begin{equation}\label{eq:definitiondeltas}
  	\bm \Delta_0 = 1-\bm c +\bm a,\quad \bm \Delta_1 = \bm b-\bm a,\quad \bm \Delta_2 = \bm c - \bm b,\quad
  	\bm X_0 =\frac  {\bm a}{\bm \Delta_0},  \quad \bm B = \bm s.
  	\end{equation}
 
  	We may now cut the excursion $\Exc$ into three excursions, in the manner prescribed by Aldous \cite{Aldous94}.
  	
  	\begin{equation}\begin{aligned}\label{eq:definitiontheta}
  	 \eta_0(x) &=  \bm\Delta_0 \,x + (1-\bm\Delta_0)\mathbf 1_{[x>\bm X_0]},\ x\in[0,1], \quad
  	 &\Exc_0 &= \frac 1 {\sqrt{\bm \Delta_0}} \, \Exc \circ \eta_0\\
  	 \eta_1(x) & =  \bm a + \bm\Delta_1\,x,\ \ x\in[0,1],\quad
  	 &\Exc_1 &= \frac 1 {\sqrt{\bm \Delta_1}}(\Exc \circ \eta_1 - \Exc(\bm b)        )   \\
  	\eta_2(x) & = \bm b +\bm\Delta_2\,x,\ \ x\in[0,1], \quad
  	&\Exc_2 &= \frac 1 {\sqrt{\bm \Delta_2}}(\Exc \circ \eta_2 - \Exc(\bm b))
  	\end{aligned}\end{equation}
  	Then  \cite[Corollary 3]{Aldous94} states that the random functions $\Exc_0,\Exc_1,\Exc_2$ are three independent Brownian excursions, independent from the vector $(\bm \Delta_0,\bm \Delta_1,\bm \Delta_2)$;
	moreover, the latter has distribution $\Dirichlet(1/2,1/2,1/2)$.

  	In addition the piecewise affine maps $(\eta_i)_{0\leq i \leq 2}$ naturally put the local minima of $\Exc$ (except the one at $t=\bm b$) in bijection with the disjoint union of the local minima of $\Exc_0$, $\Exc_1$ and $\Exc_2$. 
  	(Indeed, almost surely, $\Exc$ does not have a local minimum at $t= \bm a$ nor at $t = \bm c$.)
  	In particular, this implies (as shown in the proof of Theorem 1.6 in \cite{MickaelConstruction}, see in particular Observations 5.2 and 5.3 there) that
        \begin{itemize}
        \item $\bm B$ is a Bernoulli($1-p$) random variable,
        \item there exist three independent i.i.d. sequences of Bernoulli($1-p$) random variables \linebreak $\bm S^p_0,\bm S^p_1,\bm S^p_2$ such that for $0\leq x<y\leq 1$ and $k\in \{0,1,2\}$, 
  	\begin{equation}  \label{eq:DecSubExc}
  	\Dec(x,y;\Exc_k,\bm S^p_k) = \Dec(\eta_k(x),\eta_k(y);\Exc,\bm S^p),
  	\end{equation}
      \item the random variables $\Exc_0,\Exc_1,\Exc_2,(\bm \Delta_0,\bm \Delta_1,\bm \Delta_2),\bm B$ and the sequences $\bm S^p_0,\bm S^p_1,\bm S^p_2$ are all independent.
        \end{itemize}
        
	Now, let $\bm W^p$ be the Brownian cographon (of parameter $p$) associated with $(\Exc,\bm S^p)$ (see \cref{def:BrownianCographon}). 
	Similarly, for $k\in \{0,1,2\}$, consider $\bm W_k^p$ the Brownian cographon (also of parameter~$p$)
	 constructed from $(\Exc_k,\bm S^p_k)$, {\emph i.e.}  $$
	\begin{array}{ r c c l}
	\bm W^p_k: & [0,1]^2 &\to& \{0,1\}\\
	& (x,y) & \mapsto & \Dec(x,y;\Exc_k,\bm S^p_k).
	\end{array}
	$$
	These three random graphons form a triple of i.i.d. random graphons, independent from the random variables $\bm B$ and $(\bm \Delta_0,\bm \Delta_1,\bm \Delta_2)$.
	
	Let $\bm I$ be an independent set of $\bm W^p$. 
	For each $k=0,1,2$, denote by $\bm A_k$ the image of $[0,1]$ by $\eta_k$, namely, $\bm A_0 = [0,\bm a] \cup [\bm c,1]$, $\bm A_1 = [\bm a, \bm b]$ and $\bm A_2 = [\bm b, \bm c]$. 
	Define $\bm I_k$ as follows: 
	\begin{equation*}
	\bm I_k =\eta_k^{-1} (\bm I \cap \bm A_k) \ ,\ \ k\in \{0,1,2\}.
	\end{equation*}
	Since the images of the affine injective maps $(\eta_k)_{k \in \{0,1,2\}}$
	 partition $[0,1]$ up to measure-negligible overlaps,
	\begin{equation}\label{eq:LebSubExc}
	\Leb(\bm I) = \bm \Delta_0 \Leb(\bm I_0) + \bm \Delta_1 \Leb(\bm I_1) + \bm \Delta_2 \Leb(\bm I_2).
	\end{equation}
	Since $\bm I$ is an independent set of $\bm W^p$,
	\cref{eq:DecSubExc} implies that $\bm I_k$ is an independent set of $\bm W_k^p$ for every $k\in \{0,1,2\}$.
	In particular, $ \Leb(\bm I_k) \leq \ind(\bm W_k^p).$
	Moreover, we notice that if $\bm B = 1$, then either $\Leb(\bm I_1) = 0$ or $\Leb(\bm I_2) = 0$ (by definition of independent set in a graphon).
        Together with \cref{eq:LebSubExc}, we deduce
	\begin{equation*}
	\Leb(\bm I) \leq  \bm \Delta_0\ind(\bm W_0^p) +\bm B \max\Big(\bm \Delta_1\ind(\bm W_1^p),\bm \Delta_2\ind(\bm W_2^p)\Big) + (1-\bm{B})\Big(\bm \Delta_1\ind(\bm W_1^p)+\bm \Delta_2\ind(\bm W_2^p)\Big).
	\end{equation*}
	From \cref{eq:DefInd},
        taking the supremum over independent sets $\bm I$ of $\ind(\bm W^p)$, one obtains the following a.s. inequality
	\begin{equation*}
	\ind(\bm W^p)\leq  \bm \Delta_0\ind(\bm W_0^p) +\bm B \max\Big(\bm \Delta_1\ind(\bm W_1^p),\bm \Delta_2\ind(\bm W_2^p)\Big) + (1-\bm{B})\Big(\bm \Delta_1\ind(\bm W_1^p)+\bm \Delta_2\ind(\bm W_2^p)\Big).
	\end{equation*}
	Since $\ind(\bm W_k^p)$ has the same distribution as $\ind(\bm W^p)$ for $k\in \{0,1,2\}$,
	and the three are independent,
	the right-hand-side is a random variable distributed as ${\bm Y_{(p)}^{\Law(\ind(\bm W^p))}}$,
	proving that $\Law(\ind(\bm W^p))$ satisfies \cref{eq:TheDistributionalInequality}.	
	\end{proof}

  \subsection{Solving the inequation}
  \label{ssec:FixedPoint}

  \begin{proposition}
    \label{prop:SolvingTheInequation}
    For $p$ in $[0,1)$, the Dirac distribution $\mu=\delta_0$ is the only probability distribution on $[0,1]$
    solution of the inequation \eqref{eq:TheDistributionalInequality}.
  \end{proposition}
  We start by stating and proving a key lemma.
    Recall the definition of $\Phi_p(\mu)$ from \eqref{eq:TheDistributionalInequality}.
    The map $\Phi_p$ is a functional from the space $\mathcal M_1([0,1])$ of probability
    distributions on $[0,1]$.
    The space $\mathcal M_1([0,1])$ can be endowed with the so-called {\em Wasserstein distance} 
    (also called optimal cost distance, or Kantorovich--Rubinstein distance):
    \[ d_W(\nu,\nu') = \inf_{\bm X,\bm X': \bm X \sim \nu, {\bm X'} \sim \nu'} \ex[|\bm X- {\bm X'}|],\]
    where the infimum is taken over all pairs $(\bm X, {\bm X'})$ of random variables defined on the same probability space
    with distributions $\nu$ and $\nu'$, respectively.
    We will use below the fact that this infimum is reached (for an explicit expression of the minimizing coupling see e.g. Remark 2.30 in \cite{Peyre}).

    Furthermore since we are working on a compact space,
    convergence for $d_W$ is equivalent to weak convergence of measures (see \cite[Sec.6]{Villani}).
    \begin{lemma}
      For $p\in [0,1)$, the map $\Phi_p$ is a weak contraction for $d_W$,
       {\emph i.e.} for measures $\mu$ and~$\nu$ in $\mathcal M_1([0,1])$
      with $\mu \ne \nu$, we have $d_W(\Phi_p(\mu),\Phi_p(\nu)) < d_W(\mu,\nu)$.
      \label{lem:contraction}
    \end{lemma}
    \begin{proof}
      Let $\mu$ and $\nu$ be probability distributions on $[0,1]$.
      We choose a pair $(\bm X^\mu_0, \bm X^\nu_0)$ of random variables of distribution
      $\mu$ and $\nu$ respectively such that $\ex[|\bm X^\mu_0- \bm X^\nu_0|]=d_W(\mu,\nu)$
      (as mentioned above, such a coupling always exists).
      We then let $(\bm X^\mu_1, \bm X^\nu_1)$ and $(\bm X^\mu_2, \bm X^\nu_2)$ be independent copies
      of $(\bm X^\mu_0, \bm X^\nu_0)$. Finally, we let $(\bm \Delta_0, \bm \Delta_1, \bm \Delta_2)$ be
      a random vector with distribution $\Dirichlet(1/2,1/2,1/2)$ independent from
      $(\bm X^\mu_i, \bm X^\nu_i)_{i \in \{0,1,2\}}$, and $\bm B$ a Bernoulli$(1-p)$ random variable, independent from $(\bm \Delta_0, \bm \Delta_1, \bm \Delta_2,(\bm X^\mu_i,\bm X^\nu_i)_{i \in \{0,1,2\}})$.
      
     As in Eqs. (\ref{eq:DefY0}) - (\ref{eq:DefY}), we define ${\bm Y_{0}^\mu}$, ${\bm Y_{0}^\nu}$, ${\bm Y_{1}^\mu}$, ${\bm Y_{1}^\nu}$, ${\bm Y_{(p)}^\mu}$ and ${\bm Y_{(p)}^\nu}$ on the same probability space and coupled in a non-trivial way: 
we use the same vector $(\bm \Delta_0, \bm\Delta_1, \bm \Delta_2)$ and Bernoulli variable $\bm B$ for both ${\bm Y_{(p)}^\mu}$ and ${\bm Y_{(p)}^\nu}$.

      Then we have
      \begin{multline}
        \ex\big[ |{\bm Y^\mu_0} - {\bm Y^\nu_0} | \big]
      \le \sum_{i=0}^2 \ex\big[\bm \Delta_i\big] \ex\big[ |\bm X^\mu_i- \bm X^\nu_i|\big]
      =\left( \sum_{i=0}^2 \ex\big[\bm \Delta_i\big] \right) 
      d_W(\mu,\nu) =d_W(\mu,\nu),
      \label{eq:IneqDiffY0}
    \end{multline}
      where we used successively the fact that $ \bm  \Delta_i$ is independent from $(\bm X^\mu_i,\bm X^\nu_i)$,
      the fact that the coupling $\bm X^\mu_i,\bm X^\nu_i$ minimizes their $L^1$ distance
      and the fact that $\sum_{i=0}^2 \bm \Delta_i=1$ almost surely.
     We also have
      \begin{equation}
        \ex\big[ |{\bm Y^\mu_1} - {\bm Y^\nu_1} | \big] 
      \le \ex\big[\bm \Delta_0\big] \ex\big[ |\bm X^\mu_0- \bm X^\nu_0|\big] 
      + \ex\Big[ \big|\max(\bm \Delta_1 \, \bm  X^\mu_1, \bm \Delta_2 \, \bm X^\mu_2) 
      - \max(\bm \Delta_1 \,  \bm X^\nu_1, \bm \Delta_2 \, \bm X^\nu_2) \big|\Big]. 
      \label{eq:UpperBoundDiffY1}
 \end{equation}
      We recall the trivial inequality $|\max(a,b)-\max(c,d)| \le \max(|a-c|,|b-d|) \le |a-c|+|b-d|$.
      Besides, the second inequality is strict as soon as $a \ne c$ and $b\ne d$.
      Taking 
      \[a=\bm \Delta_1 \,  \bm X^\mu_1, \ b= \bm \Delta_2 \, \bm X^\mu_2,\
      c= \bm \Delta_1 \,  \bm X^\nu_1, \ d=\bm \Delta_2  \, \bm X^\nu_2,\]
      we obtain that, almost surely,
      \[\big|\max(\bm \Delta_1 \,  \bm X^\mu_1, \bm \Delta_2 \, \bm X^\mu_2)             
            - \max(\bm \Delta_1 \,  \bm X^\nu_1, \bm \Delta_2 \, \bm X^\nu_2) \big| 
            \le \bm \Delta_1 |\bm X^\mu_1- \bm X^\nu_1| + \bm \Delta_2 |\bm X^\mu_2 - \bm X^\nu_2|.\]
      Moreover, since $\mu \ne \nu$, we have that $\bm X^\mu_1 \ne \bm X^\nu_1$ with positive probability.
      The same holds for $\bm X^\mu_2 \ne \bm X^\nu_2$, and, by independence, 
      both inequalities occur simultaneously with positive probability.
      Since $\bm \Delta_1$ and $\bm \Delta_2$ are positive almost surely,
      we have that $a \ne c$ and $b\ne d$ simultaneously with positive probability.
      We conclude that the above inequality is strict with positive probability.
      Taking expectation and using \cref{eq:UpperBoundDiffY1}, we get
      \begin{equation}
        \ex\big[ |{\bm Y^\mu_1} - {\bm Y^\nu_1} | \big]  < \sum_{i=0}^2 \ex\big[\bm \Delta_i\big] \ex\big[ |\bm X^\mu_i- \bm X^\nu_i|\big]
        = d_W(\mu,\nu),
        \label{eq:IneqDiffY1}
   \end{equation}
   where the last equality is taken from \eqref{eq:IneqDiffY0}.
Finally, 
 \begin{align*}d_W(\Phi_p(\mu),\Phi_p(\nu)) &\le
 \ex[|{\bm Y_{(p)}^\mu} - {\bm Y_{(p)}^\nu}|] \\ 
 &= \pr(\bm B \!=\! 1)\ex\big[|{\bm Y_{(p)}^\mu} - {\bm Y_{(p)}^\nu}| \mid \bm B \!=\!1\big] + \pr(\bm B\!=\!0) \ex\big[|{\bm Y_{(p)}^\mu} - {\bm Y_{(p)}^\nu}| \mid \bm B \!=\!0\big] \\
 &=  (1-p) \,\ex\big[ |{\bm Y^\mu_1} - {\bm Y^\nu_1} | \big]
 + p \, \ex\big[ |{\bm Y^\mu_0} - {\bm Y^\nu_0} | \big].
 \end{align*}
    The lemma thus follows
    from \cref{eq:IneqDiffY0,eq:IneqDiffY1} and the fact that $p\neq 1$.
    \end{proof}
    \begin{proof}[Proof of \cref{prop:SolvingTheInequation}]
    We first note that $\Phi_p$ is nondecreasing with respect to stochastic domination,
    namely if $\mu \le_d \nu$ then $\Phi_p(\mu) \le_d \Phi_p(\nu)$.
    Therefore if $\mu$ is a solution of Inequation~\eqref{eq:TheDistributionalInequality}, {\it i.e.} $\mu \leq_d \Phi_p(\mu)$,
    we have $\Phi_p(\mu) \le_d \Phi_p^2(\mu)$ and, iterating the application of $\Phi_p$, we get $\mu \leq_d \Phi_p(\mu) \le_d \cdots \le_d \Phi_p^k(\mu)$ for all $k \ge 1$.

    Moreover the Dirac distribution $\delta_0$ is a fixed point of $\Phi_p$.
    Since $\Phi_p$ is a weak contraction by~\cref{lem:contraction} and since $\mathcal M_1([0,1])$ is compact,
    we know from Banach fixed-point theorem that~$\Phi_p^k$ tends to $\delta_0$ in distribution.
    Combined with $\mu \leq_d \Phi_p^k(\mu)$, this forces $\mu=\delta_0$ for any probability distribution $\mu$
    on $[0,1]$ verifying \eqref{eq:TheDistributionalInequality},
    which is what we wanted to prove. 
  \end{proof}

 \subsection{Completing the proof of the sublinearity results}
\label{ssec:completing_proof} 

\cref{prop:InequationSatisfied,prop:SolvingTheInequation}
imply the following result, which is the core of the proofs
of our sublinearity results (\cref{thm:graphes_universel,thm:permutations_universel}).
  \begin{theorem}
    For $p$ in $[0,1)$, we have $\ind( \bm W^{p} )=0$ almost surely. 
    \label{thm:IndBrownianZero}
  \end{theorem}
We now proceed with the proofs of our sublinearity results.
 
 \begin{proof}[Proof of \cref{thm:graphes_universel}]
 Let $p \in [0,1)$ and consider a sequence $(\bm G_n)$ of random graphs which converges to the Brownian cographon $\bm W^{p}$. 
  By Skorokhod's representation theorem, we can represent all $\bm G_n$ and $\bm W^{p}$
  on the same probability space so that $\bm G_n$ converges to $\bm W^{p}$ in the cut distance almost surely.
  Applying \cref{prop:SemiContinuity}, we get that, a.s.,
  \[\limsup_{n\to\infty} \tfrac 1n \Ind(\bm G_n)=\limsup_{n\to\infty}\ind(W_{\bm G_n}) \le \ind( \bm W^{p} ),\]
  By \cref{thm:IndBrownianZero}, the upper bound is $0$ a.s.
 Thus, $\tfrac 1n \Ind(\bm G_n)$ converges to $0$ a.s.
    and hence in probability.
 \end{proof}
 
 \begin{proof}[Proof of \cref{thm:permutations_universel}]
Recall that, for any permutation $\sigma$,
 there is a one-to-one correspondence between increasing subsequences
  of $\sigma$ and independent sets of $\inv(\sigma)$. 
In particular, one has~$\LIS(\si)=\alpha(\inv(\si))$.

Consider now a sequence $\bm \si_n$ of random permutations tending
to the Brownian separable permuton $\bm \mu^p$ for $p \in [0,1)$.
By \cref{prop:inv}, the sequence $\inv(\bm \si_n)$
converges to the Brownian cographon $\bm W^{p}$.
Applying \cref{thm:graphes_universel} gives
that $\frac{\alpha(\inv(\bm \si_n))}{n}$ tends to $0$ in probability.
But~$\frac{\alpha(\inv(\bm \si_n))}{n}=\frac{\LIS(\bm \si_n)}{n}$ a.s.,
concluding the proof.
 \end{proof}

\section{Expected number of independent sets of linear size}\label{sec:EstimationDrmota}

For $k\leq n$ let $\bm X_{n,k}$ be the random variable given by the number of independent sets of size $k$ in a
uniform labeled cograph of size $n$.
The goal of this section is to prove \cref{th:AsymptX_nalpha},
{\emph i.e.} to estimate $\mathbb{E}[\bm X_{n,k}]$ in the case where $k$ grows linearly in $n$.

The first step of the proof is to obtain equations for the exponential generating series of cographs with a marked independent set,
through symbolic combinatorics. To this aim, it is convenient to encode cographs by their \emph{cotrees}. 
The asymptotic analysis is then performed via saddle-point analysis.

\subsection{Combinatorial preliminaries: cographs and cotrees}

\begin{definition}\label{def:cotree}
A labeled \emph{cotree} of size $n$ is a rooted tree $t$ with $n$ leaves labeled from $1$ to $n$ such that:
\begin{itemize}
\item $t$ is not plane, (\emph{i.e.} the children of every internal node are not ordered);
\item every internal node has at least two children;
\item every internal node in $t$ is decorated with a $\Zero$ or a $\One$; 
\item decorations $\Zero$ and $\One$ should alternate along each branch from the root to a leaf.
\end{itemize}
An \emph{unlabeled cotree} of size $n$ is a labeled cotree of size $n$ where we forget the labels on the leaves.
\end{definition}


For an unlabeled cotree $t$, we denote by $\cograph(t)$ the unlabeled graph defined recursively as follows (see an illustration in \cref{fig:ex_cotree}):
\begin{itemize}
\item If $t$ consists of a single leaf, then $\cograph(t)$ is the graph with a single vertex.
\item Otherwise, the root of $t$ has decoration $\Zero$ or $\One$ and has subtrees $t_1$, \dots, $t_d$
attached to it ($d \ge 2$).
Then, if the root has decoration $\Zero$, we let $\cograph(t)$ be the {\em disjoint union}
of $\cograph(t_1)$, \dots, $\cograph(t_d)$. 
Otherwise, the root has decoration $\One$, and
 we let \linebreak $\cograph(t)$ be the {\em join} 
of $\cograph(t_1)$, \dots, $\cograph(t_d)$.
\end{itemize}

\begin{figure}%
\begin{center}
\includegraphics[width=8cm]{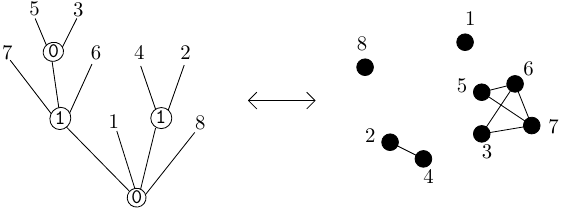}
\caption{Left: A labeled cotree $t$ with $8$ leaves. Right: The associated labeled cograph $\cograph(t)$ of size $8$. }
\label{fig:ex_cotree}
\end{center}
\end{figure}

Note that the above construction naturally entails a one-to-one correspondence 
between the leaves of the cotree $t$ and the vertices of its associated graph $\cograph(t)$. 
Therefore, it maps the size of a cotree to the size of the associated graph. 
Another consequence is that we can extend the above construction to a \emph{labeled} cotree $t$, 
and obtain a \emph{labeled} graph (also denoted $\cograph(t)$), with vertex set $\{1,\dots,n\}$: 
each vertex of $\cograph(t)$ receives the label of the corresponding leaf of $t$. 

By construction, for all cotrees $t$, the graph $\cograph(t)$ is a cograph.
Conversely, each cograph can be obtained in this way, and this correspondence is one-to-one. 
This property is ensured by the alternation of decorations $\Zero$ and $\One$ in cotrees. 
This was first shown in ~\cite{corneil}. 
The presentation of~\cite{corneil}, although equivalent, is however a little bit different, 
since cographs are generated using exclusively ``complemented unions''
instead of disjoint unions and joins. 
The presentation we adopt has since been used in many algorithmic papers, see \emph{e.g.}~\cite{Habib,Bretscher}. 

From a cograph $G$, the unique cotree $t$ such that $\cograph(t)=G$ is recursively built as follows. 
If $G$ consists of a single vertex, $t$ is the unique cotree with a single leaf. 
If $G$ has at least two vertices, we distinguish cases depending on whether $G$ is connected or not. 
\begin{itemize}
 \item If $G$ is not connected, the root of $t$ is decorated with $\Zero$ and the subtrees attached to it are the cographs associated with the connected components of $G$. 
 \item If $G$ is connected, the root of $t$ is decorated with $\One$ and the subtrees attached to it are the cographs associated with the induced subgraphs of $G$ 
whose vertex sets are those of the connected components of $\bar{G}$, where $\bar{G}$ is the complement of $G$ 
(graph on the same vertices with complement edge set). 
\end{itemize}
Important properties of cographs which justify the correctness of the above construction are the following: 
cographs are stable under taking induced subgraphs and complement, and a cograph~$G$ of size at least two is not connected exactly when its complement $\bar{G}$ is connected.

\begin{remark}\label{rem:alpha_vs_omega}
The transformation which switches every decoration $\One \leftrightarrow \Zero$
 in a cotree is of course an involution. 
 Moreover, it turns independent sets into cliques in the corresponding cograph
 (indeed $\{v,v'\}$ is an edge in $\cograph(t)$ if and only if the first common ancestor of the corresponding leaves of $t$ has decoration $1$).
 This proves that for every $n$, 
 if $\bm G_n$ denotes a uniform random cograph (either labeled or unlabeled) of size $n$, then 
\begin{equation}\label{eq:OmegaEgalAlpha}
\alpha({\bm G_n})\stackrel{(d)}{=} \omega({\bm G_n}),
\end{equation}
where $\omega(G)$ is the maximum size of a clique in the graph $G$.
\end{remark}

\subsection{Proof of \cref{th:AsymptX_nalpha}: Enumeration}\label{Subsec:Enumeration}

Let $\mathcal{L}$ be the combinatorial family of labeled cotrees for which we forget decorations, counted by the number of leaves. 
Let $L(z)$ denote the corresponding exponential generating function. The series $L(z)=\sum_{\ell \in \mathcal{L}} z^{|\ell|}/|\ell|!$ is the unique formal power series solution of
\begin{equation}\label{eq:SerieS}
L(z)=z+e^{L(z)}-1-L(z)
\end{equation}
such that $L(0)=0$.
(The enumeration of $\mathcal{L}$ is provided in \cite[Example VII.12 p.472]{Violet} under the name of \emph{labeled hierarchies}, see also Propositions 5.1 and 5.4 of \cite{CographonBrownien}.)

Next we consider pairs $(G,I)$, where $G$ is a (labeled) cograph and
$I$ an independent set of~$G$. We see such a pair as a \emph{marked} cograph.
Let us consider the associated bivariate generating function
\[
C(z,u) = \sum_{\substack{(G,I)\,:\, G\  \mbox{cograph},\,\\ I\subseteq V_G\ \mbox{independent}}} \frac 1{|V_G|!} z^{|V_G|} u^{|I|}.
\]

Then $\mathbb{E}[\bm X_{n,k}] = \frac{ [z^n u^{k}]\, C(z,u) }{ [z^n] C(z,0) }$.
Our goal is then to find the asymptotics of these coefficients.

Note that if $G$ is reduced to a single vertex $\bullet$  we have $(G,I)=(\bullet,\varnothing)$ or $(\bullet,\{\bullet\})$, therefore
\begin{equation}\label{eq:C_C0_C1}
C(z,u) = z + zu + C_0(z,u) + C_1(z,u),
\end{equation}
where $C_0(z,u)$ (resp $C_1(z,u)$) is the bivariate series of the set $\mathcal{C}_0$ (resp. $\mathcal{C}_1$) of 
marked cographs (necessarily of size $\geq 2$) for which the root of the associated cotree is decorated with a $0$ (resp. a $1$).
We have $$L(z) = z+ C_0(z,0) = z+ C_1(z,0).$$
Indeed when the decoration of the root is fixed, the other decorations are then determined by the alternation condition.

\begin{proposition}[Functional equations for $C_0,C_1$]\ 
\begin{enumerate}
\item A relation between the series $C_0(z,u)$ and $C_1(z,u)$ is given by
\begin{equation}\label{eq:C_1}
C_0(z,u) = e^{z(1+u)+C_1(z,u)}-1 -z(1+u)-C_1(z,u),
\end{equation}
\item and the series $C_1(z,u)$ is a solution of
\begin{equation}\label{eq:C_0}
C_1(z,u) = e^{L(z)}-1-L(z) + (e^{L(z)}-1)\left( e^{z(1+u)+ C_1(z,u)} - 1 -C_1(z,u) -L(z) \right).
\end{equation}
\end{enumerate}
\end{proposition}

In the proof below, we make use of the notation $\exp_{\ge k} (x) := \sum_{i \ge k} \frac{x^i}{i!}$.

\begin{proof}
  When a cotree $T$ has its root $r$ decorated by a $0$, if we denote by $(T_i)$ the subtrees rooted at the children of $r$, 
  then the cograph $G$ associated with $T$ is the disjoint union of the cographs $G_i$ corresponding to the $T_i$.
  An independent set of $G$ is then the union of independent sets chosen in each of the $G_i$. 
  Recall also that by definition of cotrees, $r$ has at least two children. 
  
  Therefore the marked cographs for which the root of the associated cotree is decorated with a $0$ 
  can be described as a multiset of at least two elements chosen between $(\bullet,\varnothing)$, $(\bullet,\{\bullet\})$ and the elements of $\mathcal{C}_1$.

  Using the symbolic method for labeled structures \cite{Violet}, we get the equation 
\begin{align*}
  C_0(z,u) &= \exp_{\ge 2} \left( z + zu + C_1(z,u) \right)
  = e^{z(1+u)+C_1(z,u)}-1 -z(1+u)-C_1(z,u),
\end{align*}
which is \cref{eq:C_1}.

When on the contrary a cotree $T$ has its root $r$ decorated by a $1$, 
if we denote again by $(T_i)$ the subtrees rooted at the children of $r$, 
then the cograph $G$ associated with $T$ is the join of the cographs $G_i$ corresponding to the $T_i$.
An independent set of $G$ must then be an independent set chosen in one of the $G_i$ only 
(and the other children of $r$ do not contribute to this independent set). 

Let $\mathcal{C}_{\varnothing}$ denote the set of cographs without mark and whose cotree does not have a root decorated by a $1$, 
{\emph i.e.} $\mathcal{C}_{\varnothing}$ is the set consisting in $(\bullet,\varnothing)$ and the elements of $\mathcal{C}_0$ marked with an empty independent set.
Then, we distinguish two cases to describe the elements of $\mathcal{C}_1$
(marked cographs for which the root of the associated cotree is decorated with a $1$). 
Either they are marked with an empty independent set, 
and in this case they can be described as multisets of at least two elements of $\mathcal{C}_{\varnothing}$. 
Or they are marked with a nonempty independent set, and 
they can be described as the pairs consisting of 
\begin{itemize}
 \item a cograph which is either $(\bullet, \{\bullet\})$ or an element of $\mathcal{C}_0$ marked with a nonempty independent set
 (for the graph $G_i$ containing the independent set); and 
 \item a multiset of at least one element of $\mathcal{C}_{\varnothing}$ (for the other graphs $G_i$). 
\end{itemize} 

We get the equation
\begin{align*}
C_1(z,u) &= \exp_{\ge 2}\left( z + C_0(z,0) \right) + 
\left( zu + C_0(z,u) - C_0(z,0)\right) \times \exp_{\ge 1} \left( z + C_0(z,0) \right). 
\end{align*}
Thus, by eliminating $C_0(z,u)$ and using that $L(z)=z+C_0(z,0)$, we obtain \cref{eq:C_0}.
\end{proof}

From \cref{eq:C_C0_C1,eq:C_1} we have
\begin{equation}\label{eq:C_C1}
C(z,u) = e^{z(1+u)+C_1(z,u)}-1.
\end{equation}
In the following, to get the asymptotics of the coefficients of $C(z,u)$,
we study $C_1(z,u)$ using \cref{eq:C_0}.

\subsection{Proof of \cref{th:AsymptX_nalpha}: main asymptotics - Proof of Eq.~\eqref{eq:AsymptotiqueEX_n,k}} 

Following Flajolet and Sedgewick \cite[p. 389]{Violet}, we say that
a domain $\Delta$ is a {\em $\Delta$-domain at $\rho$} if there exist two real numbers  $R>\rho$ and $0<\phi<\tfrac{\pi}{2}$ such that
$$\Delta=\{z \in \mathbb{C} \mid |z|<R,\, z\neq \rho,
|\arg(z-\rho)|>\phi\}$$
and that a power series is {\em $\Delta$-analytic}
if it is analytic in some $\Delta$-domain at $\rho$, where $\rho$ is its radius of convergence.

From \cite[Example VII.12 p.472]{Violet} the series $L(z)$ has radius of convergence \linebreak ${\rho=2\log(2)-1}$ and is $\Delta$-analytic. 
Moreover the following expansion holds in a $\Delta$-domain at $z=\rho$:
\begin{equation}\label{eq:Asympt_S}
L(z)\underset{z\to \rho}{=}\log(2)-\sqrt{\rho}\sqrt{1-\tfrac{z}{\rho}} +\mathcal{O}(1-\tfrac{z}{\rho}).
\end{equation}
\cref{eq:Asympt_S} combined with the transfer theorem \cite[Cor.VI.1]{Violet} yields
\begin{equation}\label{eq:Asympt_ell}
[z^n]L(z) \sim \sqrt{\frac{2\log 2-1}{4\pi}}\rho^{-n}\,n^{-3/2}.
\end{equation}

This allows to obtain the asymptotics of $[z^n] C(z,0)$. To get the one of $[z^n u^{k}]\, C(z,u)$,
we turn to the study of $C_1(z,u)$.


Fix $u \in \mathbb{C}$. The overall strategy is to perform saddle-point analysis with $C_1(z,u)$.
To do so we rewrite \cref{eq:C_0} as $C_1(z,u)$ is solution of $c= G(z,c,u)$ where
\[
G(z,c,u) = e^{L(z)}-1-L(z) + (e^{L(z)}-1)\left( e^{c + z(1+u)} - 1 -c -L(z) \right).
\]

We will show that this almost fits the settings of the so-called \emph{smooth implicit-function schema} (see \cite[Sec. VII.4.1]{Violet}), 
only the nonnegativity of the coefficients of $G$ is not satisfied here. 
Nevertheless, we shall prove that sufficient conditions for the validity of \cite[Thm. VII.3 p.468]{Violet} are satisfied.
First
observe that for every $u \in  \mathbb{C}$ the bivariate series $(z,c)\mapsto G(z,c,u)$ is analytic for $|z|< \rho$ and~$c\in \mathbb{C}$.

\subsubsection{Solution of the characteristic system}
We use the notational convention that, for any function $H$ and variable $t$, 
$H_t$ denotes the partial derivative of $H$ with respect to $t$. 
We consider the \emph{characteristic system}
\begin{equation}
\label{eq:CharacteristicSystemGenerale}
G(r,s,u)=s,\qquad G_c(r,s,u)=1,
\end{equation}
namely
\begin{align}
  e^{L(r)}-1-L(r) + (e^{L(r)}-1)\left( e^{s + r(1+u)} - 1 -s -L(r) \right)&=s, \label{eq:CharacteristicSystem}\\
(e^{L(r)}-1) \left( e^{s + r(1+u)}  - 1 \right)&=1. \label{eq:CharacteristicSystem2}
\end{align}
We aim at proving that, for any $u>0$, \eqref{eq:CharacteristicSystemGenerale} admits a unique solution $(r,s) =(r(u),s(u))$ with $0<r<\rho$ and $0<s$. 
Below, we often use that the radius of convergence $\rho$ of $L(z)$ satisfies $\rho = 2\log(2)-1$ and $L(\rho) = \log(2)$. 

We observe that if we substitute \cref{eq:CharacteristicSystem2} into \cref{eq:CharacteristicSystem} we obtain that
\begin{equation}\label{eq:s(u)}
s=1-L\left(r \right).
\end{equation}
Then \cref{eq:CharacteristicSystem2} can be rewritten as
\begin{equation}\label{eq:r(u)}
  F(r,u)=0 \quad \mbox{with} \quad F(x,u)=(e^{L(x)}-1)(e^{1-L(x)+x(1+u)}-1)-1.
\end{equation}

We have
\begin{align*}
F_x(x,u)&=\left(L'(x)e^{L(x)}+\big(1+u-L'(x)\big)(e^{L(x)}-1)\right)e^{1-L(x)+x(1+u)}-L'(x)e^{L(x)}\\
&= (1+u)(e^{L(x)}-1)e^{1-L(x)+x(1+u)} +L'(x) \big(e^{1-L(x)+x(1+u)} -e^{L(x)}\big).
\end{align*}
Fix $u>0$. 
For $0 < x\leq\rho$, one has 
\[1-L(x)+x(1+u) > 1-L(x)+x \ge L(x)\] 
(indeed one has equality for $x=\rho$ and $2L(x)-x$ is increasing), so that $F_x(x,u) > 0$.
Therefore, the function $F$ is increasing with $x$ on the interval $[0,\rho]$. 
Since $F(0,u)=-1$ and $F(\rho,u) > F(\rho,0)=0$, for any $u>0$ \cref{eq:r(u)} admits a unique solution $r=r(u)$
in $[0,\rho]$, and we have $0<r<\rho$.

From \cref{eq:s(u)}, we have $s=1-L(r)$.
Since $L$ is increasing and $r<\rho$, we have $s >1-L(\rho) = 1-\log(2)>0$.

We conclude that for $u>0$, the characteristic system \eqref{eq:CharacteristicSystemGenerale}
has a unique solution $r(u),s(u)$ in $[0,\rho] \times \mathbb{C}$,
and we have $0<r(u)<\rho$ and $s(u)>0$.
In particular, $r(u),s(u)$ belongs to the analyticity domain of $G$.

\subsubsection{Locating the singularity of $C_1(z,u)$}
Fix $u>0$. To obtain the singular behavior of $C_1(z,u)$ as in~\cite[Thm. VII.3 p.468]{Violet} 
despite the negativity of some coefficients of $G$, 
we see from \cite[Note VII.16 p.471]{Violet} that it is enough to show the following: 
$C_1(z,u)$ has radius of convergence $r(u)$
and its value at this singularity is given by $C_1(r(u),u)=s(u)$,
\emph{i.e.} the dominant singularity of $C_1(z,u)$ corresponds to the solution of the characteristic system.

The argument to prove this is an adaptation of that in the proof of \cite[Thm. VII.3 p.468]{Violet}
to our setting where $G$ has some negative coefficients but a larger analyticity region than what is usually assumed. 
Namely, our $G$ is analytic on the whole domain $\{|z|< \rho, c \in \mathbb{C}\}$,
while the smooth implicit-function schema only assumes analyticity on $\{|z|<R, |c|<S\}$ for some $R, S > 0$ (with the notation of \cite[Sec.~VII.4.1]{Violet}).
Let us denote temporarily $\rho(u)$ the radius of convergence of $C_1(z,u)$,
which is a singularity of $C_1(z,u)$ from Pringsheim's theorem.


We first show that $\rho(u) \ge r(u)$. We proceed by contradiction, and assume $\rho(u)<r(u)$.
We set $\sigma(u)=C_1(\rho(u),u)$ and distinguish two cases.
\begin{itemize}
\item Assume $\sigma(u)<+\infty$. Then, since $C_1(z,u)$ is a solution of $c=G(z,c,u)$,
 we have $\sigma(u)=G(\rho(u),\sigma(u),u)$.
 By uniqueness of the solution of the characteristic system \eqref{eq:CharacteristicSystemGenerale},
 we necessarily have $G_c(\rho(u),\sigma(u),u)\ne1$.
Therefore, using the analytic implicit function lemma \cite[Lemma VII.2, p.469]{Violet},
$C_1(z,u)$ can be extended analytically in a neighborhood of $\rho(u)$,
contradicting the fact that $\rho(u)$ is a singularity of $C_1(z,u)$.
\item If $\sigma(u)=+\infty$, one checks easily that the function $z \mapsto G_c(z,C_1(z,u),u)$
tends to $+\infty$ when $z$ tends to $\rho(u)$.
But for $z=0$, we have $G_c(0,C_1(0,u),u)=0$. 
The intermediate value theorem ensures the existence of $z_1$ in $(0,\rho(u))$ such
that $G_c(z_1,C_1(z_1,u),u)=1$. This gives an other solution $(z_1,C_1(z_1,u))$ of the characteristic system,
contradicting the uniqueness of the solution.
\end{itemize}
We have reached a contradiction in both cases, proving that $\rho(u) \ge r(u)$.


This allows us to consider $C_1(r(u),u)$ (which is possibly infinite), 
and we assume for the sake of contradiction that $C_1(r(u),u) \ne s(u)$.
Then for $a<r(u)$ sufficiently closed to $r(u)$ the equation $y=G(a,y,u)$ admits several solutions $y \in \mathbb{C}$:
\begin{itemize}
 \item one is given by $y=C_1(a,u)$,
 \item and two are obtained evaluating in $a$ the two functions $y_1(z)$ and $y_2(z)$ 
 given by the singular implicit function lemma \cite[Lemma VII.3, p.469]{Violet} applied to the point $(r(u),s(u))$. 
\end{itemize} 
(Note that the applicability of this lemma is guaranteed by the fact that 
$(r(u),s(u))$ is a solution of the characteristic system and \cref{eq:hyp_check_derivatives_cc,eq:hyp_check_derivatives_u} below.)

From \cite[Lemma VII.3, p.469]{Violet}, 
it is clear that the last two solutions above are distinct for $a$ close enough to $r(u)$. 
The first one is also different from them for $a$ close enough to $r(u)$: 
indeed, for $z$ tending to $r(u)$, $C_1(z,u)$ tends to $C_1(r(u),u)$ while the two other solutions tend to~$s(u)$.
However, the function $y \mapsto G(a,y,u)$ is strictly convex 
(one checks easily that its second derivative is positive)
and therefore cannot cross three times the main diagonal. 
We have reached a contradiction. 
We conclude that $C_1(r(u),u)=s(u)$.

 
It remains to prove $\rho(u)=r(u)$.
Since $(r(u),s(u))$ is a solution of the characteristic system,
there is no analytic solution of the equation $y=G(z,y,u)$ around the point $(z,y)=(r(u),s(u))$
(see the proof of \cite[Lemma VII.3, p.469]{Violet}, where it is shown that any solution $y$ 
has a series expansion involving a square-root term and hence cannot be analytic).
Therefore $C_1(z,u)$ cannot be extended analytically to a neighborhood of $r(u)$.
So, $\rho(u)=r(u)$, as wanted.

\subsubsection{Derivatives of $G$: parametrized expressions and their signs}
Several derivatives of $G(z,c,u)$ appear in the computations below, 
to establish the asymptotic behavior of $C(z,u)$ as well as estimates (i) and (ii) of \cref{th:AsymptX_nalpha}. 
We collect useful properties of these derivatives here for convenience. 
In this paragraph, we also assume $u>0$.
Recall that 
\[G(z,c,u) = e^{L(z)}-1-L(z) + (e^{L(z)}-1)\left( e^{c + z(1+u)} - 1 -c -L(z) \right).\]

First, from the explicit expression of $G_{cc}$, it follows that 
\begin{align}
G_{cc}(r(u),s(u),u)>0. \label{eq:hyp_check_derivatives_cc}
\end{align}

Moving on to $G_u(\left(r(u),s(u),u \right))$, 
it will be convenient to parametrize the involved quantities by $y:=L\left(r(u) \right)$. \cref{eq:SerieS}
becomes  
\begin{equation}\label{eq:r_y1}
y=r(u)+ e^{y} - 1-y
\end{equation}
and therefore
\begin{align}\label{eq:r_u_beta}
r(u) &= 2y + 1 - e^{y}.
\end{align}
From \cref{eq:CharacteristicSystem2}, we quickly derive 
\begin{equation}\label{eq:Simpli_r_u}
e^{s(u)+r(u)(1+u)}=\frac{e^{y}}{e^{y}-1}
\end{equation}
and we can eliminate $r(u)$ thanks to \cref{eq:r_u_beta}: we obtain
\begin{equation}\label{eq:Def_F}
u=\frac{e^{y}-2-\log(e^{y}-1)}{2y+1-e^{y}}.
\end{equation}
Next we use the definition of $G$, and then \cref{eq:Simpli_r_u,eq:r_y1}, obtaining 
\begin{align}\label{eq:GuExpr}
G_u\left(r(u),s(u),u \right) &= (e^{L(r(u))}-1)e^{s(u) + r(u)(1+u)} r(u) \nonumber \\
&= (e^{y}-1)\frac{e^{y}}{e^{y}-1} (2y+1-e^{y})
=e^{y}(2y+1-e^{y}). 
\end{align}
From the above and \cref{eq:r_u_beta}, 
we have in particular 
\begin{equation}\label{eq:GuPos}
G_u\left(r(u),s(u),u \right) > 0
\end{equation}

Finally, we focus on $G_z\left(r(u),s(u),u \right)$. Using \cref{eq:SerieS,eq:s(u)}, we start by observing that
$$
L'(z) = \frac{1}{2-e^{L(z)}}, \qquad 1+s(u)+y=2.
$$
Therefore, 
with the shorter notation $L := L(r(u))$, $L' := L'(r(u))$, $r := r(u)$, $s:= s(u)$, we have 
\begin{align}\label{GZ}
G_z\left(r(u),s(u),u \right)  &= (e^L -1) L' + e^L L' \left( e^{s + r(1+u)} - 1 -s -L \right) \\
&\qquad \qquad +  (e^L-1)
\left( e^{s + r(1+u)}(1+u) -L' \right) \nonumber\\
&= \cancel{\frac{e^{y}-1}{2-e^{y}}} + \frac{e^{y}}{2-e^{y}} 
\times\left(\frac{e^{y}}{e^{y}-1} -2\right) \\
&\qquad \qquad+  (e^{y}-1)\times \left(\frac{e^{y}}{e^{y}-1}(1+u) - \cancel{\frac{1}{2-e^{y}}}\right)\nonumber\\
&=  \frac{e^{y}}{2-e^{y}} 
\times\left(\frac{e^{y}}{e^{y}-1} -2\right) + \cancel{(e^{y}-1)}\times \left(\frac{e^{y}}{\cancel{e^{y}-1}}(1+u) \right)
\nonumber\\
&= \frac{e^{y}}{\cancel{2-e^{y}}} 
\times \frac{\cancel{2-e^{y}}}{e^{y} -1} + e^{y} (1+u)= \frac{e^{y}}{e^{y} -1} + e^{y} (1+u).
\end{align}

In particular, this gives 
\begin{align}
G_z(r(u),s(u),u)>0. \label{eq:hyp_check_derivatives_u}
\end{align}


\subsubsection{Obtaining the asymptotics}
Recall that we established that $C_1(z,u)$ has radius of convergence $r(u)$
and its value at this singularity is given by $C_1(r(u),u)=s(u)$. 
From \cite[Sec. VII.4.1]{Violet}, we therefore obtain an estimate of $C_1(z,u)$ as $z$ approaches $r(u)$.
Namely, for every $u>0$ the series $C_1(z,u)$ has a square-root singularity at $r(u)$ and in some $\Delta$-domain, we have 
\begin{equation}
\label{eq:squarerootsing}
C_1(z,u) \stackrel{z\to r(u)}{=} s(u) -\gamma_1(u)\sqrt{1-z/r(u)} + \mathcal{O}(1-z/r(u) )
\end{equation}
with $\gamma_1(u)=\sqrt{\frac{2\,r(u)\, G_z(r(u),s(u),u)} {G_{cc}(r(u),s(u),u)}}$. 
Note that $G_{cc}(r(u),s(u),u)>0$ and $G_z(r(u),s(u),u)>0$ from~\cref{eq:hyp_check_derivatives_cc,eq:hyp_check_derivatives_u}. %
The determination of the sign in front of $\sqrt{1-z/r(u)}$ uses that 
$C_1$ is increasing in $z$ when $z$ approaches $r(u)$ from the left. 

To obtain asymptotics for the coefficients of $C_1(z,u)$,
we have to extend \eqref{eq:squarerootsing} for complex $u$ around $u>0$.
We argue that the solutions $(r,s) = (r(u),s(u))$ of the characteristic system~\eqref{eq:CharacteristicSystemGenerale} have analytic continuations in a neighborhood
of every $u>0$.
Observe that $G$ is analytic
 and that the Jacobian matrix of the system is the following determinant (where all derivatives are evaluated
 at $(r(u),s(u),u)$)
\[
\left| \begin{array}{cc}  G_z & G_c-1 \\ G_{cz} & G_{cc} \end{array} \right| = 
\left| \begin{array}{cc}  G_z &  0 \\ G_{cz} & G_{cc} \end{array} \right| = G_z G_{cc}.
\]
It is nonzero for $u>0$ from \cref{eq:hyp_check_derivatives_cc,eq:hyp_check_derivatives_u}. %
Consequently, there exist analytic functions $r(u),s(u)$ defined on a neighborhood
of the positive real axis, such that, for each $u$, 
the pair $(r(u),s(u))$ is a solution of the characteristic system for such values of $u$. %

By continuity we can also ensure that, for $u$ sufficiently close to the real axis,
\begin{itemize}
\item $r(u)$ is the unique singularity of $C_1(z,u)$ of smallest modulus and $C_1(r(u),u)=s(u)$;
\item $G_z$ and $G_{cc}$ are non-zero at $(r(u),s(u),u)$.
\end{itemize}
We denote by $U$ the open set of complex numbers $u$ where these properties hold.
Therefore, as stated in \cite[Remark 2.20]{DrmotaRandomTrees}, 
it follows that the singular representation (\ref{eq:squarerootsing}) 
also holds for complex $u\in U$ (and for $z$ in a proper $\Delta$-domain depending on $u$).

Combining relation \eqref{eq:C_C1} with the above development (\ref{eq:squarerootsing})
of $C_1(z,u)$ near $z=r(u)$, we obtain for $u\in U$
$$
C(z,u) \stackrel{z\to r(u)}{=} e^{r(u)(1+u)+s(u)}-1 -\gamma(u) \sqrt{1-z/r(u)} + \mathcal{O}(1-z/r(u) ),
$$
where $\gamma(u)$ is defined by 
$$
\gamma(u)=\gamma_1(u)\exp\left(s(u)+r(u)(1+u)\right).
$$
Moreover since $C(z,u)$ is aperiodic, $r(u)$ is the unique dominant singularity of $C$ and  
\begin{equation}\label{eqznCzu}
[z^n]\, C(z,u) \stackrel{n\to +\infty}{=} \frac{\gamma(u)}{2\sqrt{\pi} }n^{-3/2} (r(u))^{-n} \left(1+\mathcal{O}(1/n)\right)
\end{equation}
uniformly for $u$ in a compact subset contained in $U$
(by Transfer Theorem \cite[Thm.VI.3]{Violet} and compactness).

Now we can proceed as in \cite[Thm.3]{DrmotaMultivariate},
with the nonnegativity of the coefficients of $G$ replaced by the above variant of the smooth-implicit function schema,
and obtain by an application of a saddle point integration
\begin{equation}\label{eq:equivalentC}
[z^n u^k] C(z,u) \sim \frac{R_{k/n}}{n^2} \left( r(u(k/n)) u(k/n)^{k/n} \right)^{-n},
\end{equation}
uniformly for $an \le k \le bn$ with $0<a<b<1$, 
where $R_\beta$ $(0< \beta < 1$) is some positive (computable) quantity and  $u = u(\beta)$ is determined by the following 
equation (which is the rewriting of \cite[(2.14)]{DrmotaMultivariate} with our notation): 
\begin{equation}\label{equ_beta}
\beta=-\frac{u\,r'(u)}{r(u)} = \frac{ u\,G_u(r(u),s(u),u) }{r(u)G_z(r(u),s(u),u)}.
\end{equation}
We explain in Remark~\ref{rmk:invertible} below why \cref{equ_beta} is indeed invertible. 


Finally with \cref{eq:Asympt_ell} we obtain
$$
\mathbb{E}[\bm X_{n,k}] = \frac{ [z^n u^{k}]\, C(z,u) }{ [z^n] C(z,0) }
\underset{\text{for } n\geq 2}{=} \frac{ [z^n u^{k}]\, C(z,u) }{ [z^n] 2L(z) }
 \sim \frac{B_{k/n}}{\sqrt{n}} (C_{k/n})^n,
$$
uniformly for $an \le k \le bn$ for some $B_\beta >0$ and with 
\begin{equation}
\label{eq:CBeta}
C_\beta:= \frac{2\log(2)-1}{r(u(\beta)) u(\beta)^{\beta}}.
\end{equation}
concluding the proof of Eq.~\eqref{eq:AsymptotiqueEX_n,k}. 

\begin{remark}\label{rmk:invertible}
Let us justify that \cref{equ_beta} can be inverted to express $u$ as a function of $\beta$. 

First, observe that \cref{eq:Def_F} defines $u$ as a function of $y$. 
This function is decreasing for $y \in (0,\log(2))$ and maps bijectively $(0,\log(2))$ to $(0,\infty)$. 
Therefore \cref{eq:Def_F} can be inverted to express $y$ as a function of $u$, which is decreasing and maps bijectively $(0,\infty)$ to~$(0,\log(2))$.

Second, from the second expression of $\beta$ in  \cref{equ_beta}, 
we obtain an expression of $\beta$ as a function of $y$, substituting \cref{eq:r_u_beta,eq:Def_F,eq:GuExpr,GZ} into \cref{equ_beta}.
This gives 
\begin{equation}\label{eq_beta_explicite}
\beta = \frac{ \left( e^{y}-2 - \log(e^{y}-1) \right) (e^{y}-1) }{ 2y + 1 - e^{y} + (e^{y}-1) (2y - 1- \log(e^{y}-1) ) }
\end{equation}
This expression defines $\beta$ as a function of $y$. 
This function is decreasing for $y \in (0,\log(2))$ and maps bijectively $(0,\log(2))$ to $(0,1)$. 

The function $\beta=\beta(u)$, obtained by composition of the above two, is therefore a bijection from $(0,\infty)$ to $(0,1)$, 
allowing to define $u=u(\beta)$. 
We observe, in addition, that $0 < u(\beta) <1$ for $\beta >0$. 
\end{remark}

\subsection{Proof of \cref{th:AsymptX_nalpha}: Estimates (i) and (ii).}\label{ssec:uinvertible}
We now analyze the expression of $C_\beta$.
Combining \cref{eq:CBeta} with \cref{eq:Def_F,eq:r_u_beta,eq_beta_explicite}, we can express $C_\beta$ as an explicit function of $y$; further inverting numerically \cref{eq_beta_explicite}
gives $C_\beta$ as a function of $\beta$.
The graph of the function $\beta \mapsto  C_\beta$ on \cref{fig:C_beta} 
was obtained in this way.

From these expressions, we can also perform Taylor expansions (see the jupyter notebook mentioned below).
The expansion of \cref{eq_beta_explicite} around $y=\log(2)$ yields
\begin{equation}\label{eq:beta}
    \beta = \frac{(y - \log(2))^2}{2\log(2)-1} +\mathcal{O}((y - \log(2))^3).
\end{equation}
Plugging this estimate in the Taylor expansion of \cref{eq:Def_F,eq:r_u_beta} around $y=\log(2)$, we obtain
\begin{align*}
r(u(\beta)) &= 2\log(2)-1 - (y - \log(2))^2 +\mathcal{O}((y - \log(2))^3) \\
&  = 2\log(2)-1+(1-2\log(2))\beta +\mathcal{O}(\beta^{3/2});\\
u(\beta)&=  \frac{2}{2\log(2)-1}(y - \log(2))^2 + \mathcal{O}((y - \log(2))^3) \\
&  = 2\beta +\mathcal{O}(\beta^{3/2}).
\end{align*}

From \cref{eq:CBeta} we deduce \cref{th:AsymptX_nalpha} item (ii):
$$
C_\beta=\frac{2\log(2)-1}{r(u(\beta)) u(\beta)^{\beta}}
= 1+\beta|\log(\beta)| +\mathrm{o}(\beta\log(\beta))
$$
when $\beta \to 0$. 

In particular, this proves $C_\beta >1$ for $\beta \in (0,\beta_0)$ for some $\beta_0 >0$.
Numerical computations give the estimate $\beta_0 \approx 0.522677\dots$; 
we furthermore observe numerically that $C_\beta$ reaches its maximum at $\beta^\star \approx 0.229285\dots$ where $C_{\beta^\star} \approx 1.3663055\dots$.

Details on the computations above are provided in a jupyter notebook embedded into this pdf
(alternatively you can download the source of the arXiv version to get the files). 
We provide both an html read-only version and an editable ipynb version for the reader's convenience.

\section{Expected number of increasing subsequences of linear size}\label{Sec:ConstantsSeparablePermutations}

We now discuss the proof of \cref{thm:expectation_permutations}, the analog of \cref{th:AsymptX_nalpha} for separable permutations. 
We start with some definitions. 

Given two permutations, $\pi$ of size $k$ and $\tau$ of size $\ell$,
the \emph{direct sum} (resp. \emph{skew sum}) of $\pi$ and $\tau$, denoted $\oplus[\pi,\tau]$ (resp. $\ominus[\pi,\tau]$) is 
the permutation $\sigma$ of size $k+\ell$ such that 
\begin{itemize}
 \item for $1 \leq i \leq k$, $\sigma(i) = \pi(i)$ (resp. $\sigma(i) = \ell + \pi(i)$), and 
 \item for $1 \leq i \leq \ell$, $\sigma(k+i) = k + \tau(i)$ (resp.  $\sigma(k+i) = \tau(i)$).
\end{itemize}
Direct sums and skew sums readily extend to more than two permutations, writing \linebreak 
$\oplus[\pi, \dots, \tau, \rho] = \oplus[\pi, \dots \oplus[\tau,\rho]]$ 
(and similarly for $\ominus$). 

As mentioned in \cref{ssec:intro_separables}, separable permutations 
are those which can be obtained from permutations of size $1$ performing direct sums and skew sums. 
This is similar to the characterization of cographs as the graphs obtained using the join and disjoint union constructions, 
from graphs with one vertex. 
And similarly to the description of cographs through their cotrees, 
this allows to associate a tree with each separable permutation.
(This is actually a special case of the construction which associate with each permutation, 
not necessarily separable, its substitution decomposition tree -- see \emph{e.g.} \cite[Section 1.1]{Nous3}). 

There are actually several presentations of this correspondence between separable permutations and trees. 
The one which is suitable here is presented in \cite[Section 2.2]{Nous1}, 
and we borrow our terminology from there.

\begin{definition}
A \emph{signed Schr\"oder tree where the signs alternate} of size $n$ 
is a rooted tree $t$ with~$n$ leaves such that: 
\begin{itemize}
 \item $t$ is plane (\emph{i.e.} the children of every internal node are ordered);
 \item every internal node has at least two children;
 \item every internal node in $t$ is decorated with $\oplus$ or $\ominus$; 
 \item decorations $\oplus$ and $\ominus$ should alternate along each branch from the root to a leaf.
\end{itemize}
\end{definition}
An important difference with cotrees is that the above trees are plane, 
while cotrees are not plane. 

We can associate to a signed Schr\"oder tree where the signs alternate 
a permutation $\mathrm{perm}(t)$ of the same size, as follows. 
\begin{itemize}
\item If $t$ consists of a single leaf, then $\mathrm{perm}(t)$ is the permutation of size $1$.
\item Otherwise, the root of $t$ has decoration $\oplus$ or $\ominus$ and has subtrees $t_1$, \dots, $t_d$
attached to it ($d \ge 2$), in this order from left to right.
Then, if the root has decoration $\oplus$, we let $\mathrm{perm}(t)$ be $\oplus[\mathrm{perm}(t_1) \dots, \mathrm{perm}(t_d)]$. 
Otherwise, the root has decoration $\ominus$, and
 we let $\mathrm{perm}(t)$ be $\ominus[\mathrm{perm}(t_1), \dots, \mathrm{perm}(t_d)]$.
\end{itemize}

\begin{proposition}\label{prop:separation_trees}
The correspondence presented above between separable permutations and signed Schr\"oder trees where the signs alternate is one-to-one. 
\end{proposition}

For a proof of this statement, 
we refer to \cite[Proposition 2.13]{Nous1} -- see also the references given in \cite{Nous1}.


We can now move to the proof of \cref{thm:expectation_permutations}. 
The strategy is the same as in the proof of \cref{th:AsymptX_nalpha}, 
using the encoding of separable permutations by their signed Schr\"oder trees where the signs alternate, 
instead of the encoding of cographs by their cotrees. 
We therefore only sketch the computations here. 
Details are provided in the attached jupyter notebook.

We denote by $S := S(z)$ the solution of 
\begin{equation}
S=z+S^2/(1-S). \label{eqS}
\end{equation} 
Equivalently, $S$ is the series of Schr\"oder trees (\emph{i.e.}, plane trees where internal nodes have at least two children) counted by leaves. 
Unlike the series $L$ in the case of cographs, the series $S$ is explicit here, namely it holds that $S(z) = ( 1+z-\sqrt{1-6z+z^2})/4$. 
Its radius of convergence is $\rho = 3-2\sqrt{2}$ and we have $S(\rho) = (2-\sqrt{2})/2$.
 
The proof also involves the generating function $S_{\ominus} := S_{\ominus}(z,u)$ (resp. $S_{\oplus}:= S_{\oplus}(z,u)$) counting separable permutations 
which can be decomposed as a skew sum (resp. direct sum) marked with an increasing subsequence.
Without marking, from \cref{prop:separation_trees}, we get 
\[S_{\ominus}(z,0)=S_{\oplus}(z,0)=\frac{S^2}{1-S}=S-z.\]
The analogs of \cref{eq:C_0,eq:C_1}  are then
\begin{align}
S_{\ominus} &= \frac{S^2}{1-S}+\left(\frac{(S_{\ominus} +z+zu)^2}{1-(S_{\ominus} +z+zu)}+zu+z -S\right)\times \left( \frac{1}{(1-S)^2}-1\right), \label{eq:Sominus}\\
S_{\oplus} &= \frac{(S_{\ominus} +z+zu)^2}{1-(S_{\ominus} +z+zu)}. \label{eq:Soplus}
\end{align}
Indeed, an increasing subsequence in a direct sum of permutations $\pi_1\oplus \dots \oplus \pi_r$
is a union of  increasing subsequences in  $\pi_1$, \dots, and $\pi_r$.
Hence elements counted by $S_{\oplus}$ can be described as sequences of at least two elements chosen between $(\bullet,\varnothing)$, $(\bullet,\{\bullet\})$ and the elements counted by~$S_{\ominus}$.
This leads to Eq.~\eqref{eq:Soplus}.
On the other hand,
 a nonempty increasing subsequence in a skew sum of permutations $\pi_1\ominus \dots \ominus \pi_r$
is an increasing subsequence in either $\pi_1$, \dots, or $\pi_r$
Therefore, elements of $S_{\ominus}$ marked with a nonempty increasing subsequence
correspond to sequences of at least two elements, 
with exactly one element counted by $zu+S_{\oplus}-S_{\oplus}(z,0)$ (either $(\bullet,\{\bullet\})$ or a $\oplus$-decomposable permutation with a nonempty marked increasing subsequence)
and other elements counted by $S$. 
We need to add a term $S_{\ominus}(z,0)=\frac{S^2}{1-S}$ for the case of an empty marked increasing subsequence.
Substituting Eq.~\eqref{eq:Soplus} and using $S_{\oplus}(z,0) = S(z)-z$ gives Eq.~\eqref{eq:Sominus}.
%
%
%
%
%
%

%
%
%
%
%
%
%
%
%
%
%
Fix $u \in \mathbb{C}$. In order to perform saddle-point analysis with $S_{\ominus}$,
we rewrite the first equation of the previous system as $S_{\ominus}= G(z,S_{\ominus},u)$ where
\begin{equation}
G(z,c,u) = \frac{S^2(z)}{1-S(z)}+\left(\frac{(c +z+zu)^2}{1-(c +z+zu)}+zu+z -S(z)\right)\times \left( \frac{1}{(1-S(z))^2}-1\right)
\label{eq:Gperm} 
\end{equation}

Again this almost fits the settings of the smooth implicit-function schema,  only the nonnegativity of the coefficients of $G$ is not verified here.  And, as we shall see, sufficient conditions for the validity of \cite[Thm. VII.3 p.468]{Violet} similar to the cograph case  are satisfied.

The bivariate series $(z,c)\mapsto G(z,c,u)$ is analytic on
$\{(z,c): |z| <\rho,\ |c+z+zu|<1\}$.

 Moreover the characteristic system, 
 \begin{equation}
\label{eq:CharacteristicSystemGeneraleBis}
G(r,s,u)=s,\qquad G_c(r,s,u)=1,
\end{equation}
 can be worked out and its solutions satisfy either
 $$(r_1,s_1)=\left(\frac{1}{1+u} (2S(r_1)- 2 \sqrt{2S(r_1)-S(r_1)^2}+1) ,-2S(r_1)+\sqrt{2S(r_1)-S(r_1)^2}\right)$$
 or
  $$(r_2,s_2)=\left(\frac{1}{1+u} (2S(r_2)+2\sqrt{2S(r_2)-S(r_2)^2}+1) ,-2S(r_2)-\sqrt{2S(r_2)-S(r_2)^2}\right).$$
 Since $s_2<0$ when $r_2>0$, we focus on solutions of the first kind.
 We claim that, for any $u>0$, there is a unique $r_1$ in $(0,\rho)$
 satisfying $$r_1=\frac{1}{1+u} (2S(r_1)- 2 \sqrt{2S(r_1)-S(r_1)^2}+1).$$
 Indeed, when $r_1$ goes from $0$ to $\rho=3-2\sqrt{2}$, the quantity $S(r_1)$
 increases from $0$ to $S(\rho) = (2-\sqrt{2})/2$ and 
 the right-hand side decreases from $1/(1+u)$ to $(3-2\sqrt{2})/(1+u)<3-2\sqrt{2}$.

 This proves that for $u>0$, the characteristic system \eqref{eq:CharacteristicSystemGeneraleBis} has a unique 
 positive solution $(r(u),s(u))$.
Moreover, with $y:=S\left(r(u) \right)$ we have  
\begin{align}\label{eq:Simpli_r_u_separables1}
r(u)&=  \frac{1}{u + 1} \left(2 y - 2 \sqrt{2y - y^2} + 1\right),\\
s(u)&= -2y+\sqrt{2y - y^2}\,,\label{eq:Simpli_r_u_separables2}\\
u &= \frac{2\left(1 - y\right) \sqrt{2y - y^2} - 1}{2 y^{2} - y}\,, \label{eq:Simpli_r_u_separables3}
\end{align}
the equation for $u$ being a consequence of the one for $r(u)$ and \eqref{eqS}
which gives $r(u) = y - \frac{y^2}{1-y}$.

One can show as in the cograph case, but comparing $S_{\ominus}(\rho(u),u)$ with $1-\rho(u)(1+u)$ instead of $+\infty$, 
that $S_{\ominus}(z,u)$ has radius of convergence $r(u)$ and that its value at this singularity is given by $S_{\ominus}(r(u),u)=s(u)$. 
Again in an analogous way to the cograph case one can verify that 
the solutions $(r,s) = (r(u),s(u))$ of the characteristic system \eqref{eq:CharacteristicSystemGeneraleBis} have analytic continuations in a neighborhood of every $u>0$, 
noting that
$$G_z(r(u),s(u),u)=\frac{6y^2\sqrt{y(2-y)}-4y^2-10y\sqrt{y(2-y)}+9y+2\sqrt{y(2-y)}-2}{y(y-1)(y-2)(2y-1)(2y^2-4y+1)}$$ is positive on the interval $(0, S(\rho))$.

Therefore since $S_{\ominus}(z,u)$ is aperiodic we can  apply \cite[Thm.3]{DrmotaMultivariate} to prove Eq.~\eqref{eq:AsymptotiqueEZ_n,k} of \cref{thm:expectation_permutations} and we obtain 
$$
E_\beta=\frac{1}{(3+2\sqrt{2})r(u(\beta)) u(\beta)^{\beta}},
$$
where $u(\beta)$ is the inverse function of \cref{equ_beta} (or rather, its permutation counterpart, with the $G$, $r$ and $s$ defined in the current section).
A complicated expression of $\beta$ in terms of~$y$ is given in the attached notebook.
This can be numerically inverted to get $y$ in terms of $\beta$ and thus to compute $E_\beta$
through 
\cref{eq:Simpli_r_u_separables1,eq:Simpli_r_u_separables2,eq:Simpli_r_u_separables3}.
One can also perform Taylor expansions around $\beta=0$. It is shown in the notebook that
\[ y=1-\tfrac{\sqrt 2}2 +\sqrt \beta \sqrt{\tfrac{3}{4} \sqrt 2 -1} +o(\sqrt \beta). \]
From there, a routine computation gives
\[r(u(\beta))=3-2\sqrt{2} +O(\beta),\qquad u(\beta)=\lambda \beta +O(\beta^2),\]
for some explicit constant $\lambda$. We finally find
\[
E_\beta
= 1+\beta|\log(\beta)| +\mathrm{o}(\beta\log(\beta)),\]
as claimed in \cref{thm:expectation_permutations}.
This implies that there exists $\beta_1>0$ (numerically estimated \linebreak at $\beta_1\approx 0.5827$) 
such that for every $\beta <\beta_1$ we have $E_\beta >1$.
\cref{thm:expectation_permutations} is proved.


\section*{Acknowledgements}

The authors are grateful to Marc Noy for stimulating discussions,
in particular for bringing to their attention the problem
of the maximum size of an independent set in a random cograph and the related literature
around the probabilistic version of the Erd\H{o}s--Hajnal conjecture. 

Thanks are also due to an anonymous referee for pointing out the existence of~\cite{Wuerfl}, and to Marc Noy and Carlos Hoppen for clarifying the status of 
the result announced by A. W{\"u}rfl in~\cite[Chapter~9]{Wuerfl}.

%



\end{document}